\documentclass[a4paper,11pt]{amsart}
\usepackage{amsmath,amsthm,amssymb,amsfonts,enumerate,color, upgreek,amsrefs,float}
\usepackage[pdftex]{graphicx}
\usepackage{pgf,tikz}
\usetikzlibrary{hobby}
\usetikzlibrary{arrows}
\usepackage[colorlinks=true, allcolors=blue]{hyperref}

\newcommand{\R}{\mathbb{R}}
\newcommand{\Rn}{\mathbb{R}^n}
\newcommand{\Z}{\mathbb{Z}}
\newcommand{\N}{\mathbb{N}}
\renewcommand{\S}{\mathcal{S}}
\newcommand{\A}{\mathcal{A}}
\newcommand{\F}{\mathcal{F}}
\newcommand{\D}{\mathcal{D}}
\newcommand{\K}{\mathcal{K}}
\renewcommand{\H}{\mathcal{H}}
\newcommand{\MA}{\operatorname{MA}}
\newcommand{\tr}{\operatorname{tr}}
\newcommand{\Int}{\operatorname{Int}}
\newcommand{\M}{\mathcal{M}}
\newcommand{\vep}{\varepsilon}
\newcommand{\supp}{\operatorname{supp}}
\newcommand{\PV}{\operatorname{PV}}

\newcommand{\diam}{\operatorname{diam}}

\newtheorem{thm}{Theorem}[section]
\newtheorem{prop}[thm]{Proposition}
\newtheorem{cor}[thm]{Corollary}
\newtheorem{lem}[thm]{Lemma}
\theoremstyle{definition}
\newtheorem{defn}[thm]{Definition}
\newtheorem{rem}[thm]{Remark}

\numberwithin{equation}{section}

\allowdisplaybreaks

\author[L.~A.~Caffarelli]{Luis A.~Caffarelli}
 
\author[M.~Soria-Carro]{Mar\'ia Soria-Carro}

    \address{Department of Mathematics\\
    The University of Texas at Austin\\
    Austin, TX 78712, USA}
    \email{caffarel@math.utexas.edu, maria.soriac@math.utexas.edu}



\thanks{Research supported by NSF grant 1500871.}

\begin{document}

\title[Integro-differential operators from $\Delta^s$ to $\MA^s$]{On a family of fully nonlinear integro-differential operators: 
From fractional Laplacian to nonlocal Monge-Amp\`{e}re}

\begin{abstract}
We introduce a new family of intermediate operators between the fractional Laplacian and the Caffarelli--Silvestre nonlocal Monge-Amp\`{e}re that are given by infimums of integro-differential operators. 
Using rearrangement techniques, we obtain  representation formulas and give a connection to optimal transport. Finally, we consider a global Poisson problem, prescribing data at infinity, and prove existence, uniqueness, and $C^{1,1}$-regularity of solutions in the full space.
\end{abstract}

\maketitle

%%%%%%%%%%%%%%%%%%%%%%%%%%%%%%%%%%%%%%%%%%%%%
\section{Introduction}
%%%%%%%%%%%%%%%%%%%%%%%%%%%%%%%%%%%%%%%%%%%%%

Integro-differential equations arise in the study of stochastic processes with jumps, such as L\'{e}vy processes. A classical elliptic integro-differential operator is the fractional Laplacian,
$$
\Delta^su(x_0)= c_{n,s} \PV  \int_{\mathbb R^n} (u(x_0+x)-u(x_0))\frac{1}{|x|^{n+2s} }dx, \qquad s\in (0,1),
$$
which can be understood as an infinitesimal generator of a stable L\'{e}vy  process. These types of processes are very well studied in probability, and their generators may be given by
$$
L_Ku(x_0)=\int_{\mathbb R^n} (u(x_0+x)-u(x_0)- x \cdot \nabla u(x_0) )K(x)dx,
$$
where the kernel $K$ is a nonnegative function satisfying some integrability condition. 

Over the last few years, there has been significant interest in studying linear and nonlinear integro-differential equations from the analytical point of view. In particular, extremal operators like 
\begin{equation} \label{eq:integro}
F u(x_0) = \inf_{K\in \mathcal K} L_K u(x_0)
\end{equation}
 play a fundamental role in the regularity theory.
 See \cite{CS2,CS3,CS4,ROS} and the references therein.
The above equation is an example of a fully nonlinear equation that appears in optimal control problems and stochastic games \cite{Krylov, Nisio}. The infimum in \eqref{eq:integro} is taken over a family of admissible kernels $\K$ that depends on the applications. 
In fact, nonlocal Monge-Amp\`{e}re equations have been developed in the form \eqref{eq:integro}, for some choice of $\K$
\cite{GS,CC,CS}. 

The Monge-Amp\`{e}re equation arises in several problems in analysis and geometry, such as the mass transportation problem and the prescribed Gaussian curvature problem \cite{DPF}. The classical equation prescribes the determinant of the Hessian of some convex function $u$:
\begin{equation*}
\det (D^2 u) = f.
\end{equation*}
In the literature, there are different nonlocal versions of the Monge-Amp\`{e}re operator that Guillen--Schwab \cite{GS}, Caffarelli--Charro \cite{CC}, and Caffarelli--Silvestre  \cite{CS} have considered. See also \cite{MS} for a  nonlocal linearized Monge-Amp\`{e}re
equation given by Maldonado--Stinga.
These definitions are motivated by the following property: if $B$ is a positive definite symmetric matrix, then
\begin{equation} \label{eq:MAidentity}
n \det (B)^{1/n}  = \inf_{A \in \A} \tr(A^T  B A ), 
\end{equation}
where $\A=\{ A \in M_n:  A>0, \ \det(A)=1\}$ and $M_n$ is the set of $n\times n$ matrices.
If a convex function $u$ is $C^2$ at a point $x_0$, then by the previous identity with $B=D^2 u(x_0)$, we may write the Monge-Amp\`{e}re operator as a concave envelope of linear operators. It follows that
\begin{equation*}
n \det(D^ 2 u(x_0))^{1/n} =\inf_{A \in \A} \Delta [u \circ A](A^{-1}x_0).
\end{equation*}

Caffarelli--Charro study a fractional version of  $\det(D^ 2 u)^{1/n}$,  replacing the Laplacian by the fractional Laplacian in the previous identity. More precisely,
$$
\mathcal{D}^s u(x_0) = \inf_{A \in \A} \Delta^s [u \circ A](A^{-1}x_0) 
= c_{n,s} \inf_{A\in \A}\ \PV \int_{\Rn} \frac{u(x_0+x)-u(x_0)}{|A^{-1}x|^{n+2s}} \, dx,
$$
where $s\in (0,1)$ and $c_{n,s} \approx 1-s$ as $s\to 1$ (see also \cite{GS}).
A different approach based on geometric considerations was given by Caffarelli--Silvestre.  In fact, the authors consider kernels whose level sets are volume preserving transformations of the fractional Laplacian kernel. Namely,
\begin{align*}
 \MA^s u(x_0) & = c_{n,s} \inf_{K\in \K_n^s} \int_{\R^n} (u(x_0+x)-u(x_0)-x \cdot \nabla u(x_0) ) K(x) \, dx,
\end{align*}
where the infimum is taken over the family,
\begin{equation} \label{eq:kernelsMA}
 \K_n^s  =\Big\{K: \R^n \to \R_+ : |\{x\in \Rn: K(x) > r^{-n-2s}\}|=|B_r| \quad \hbox{for all}~ r>0\Big\}.
\end{equation}
Notice that $|A^{-1} x|^{-n-2s} \in \K_n^s$, for any $A\in \A$. Therefore,
$$
 \MA^s u(x_0) \leq \mathcal{D}^s u(x_0) \leq \Delta^s u(x_0).
$$ 
Moreover, both $ \MA^s u$ and $\mathcal{D}^s u$ converge to $\det(D^2 u)^{1/n}$, up to some constant, as $s\to 1$.

In this paper, we introduce a new family of operators of the form,
\begin{equation} \label{eq:defop}
 \inf_{K \in \K_k^s} \int_{\Rn} (u(x_0+x)-u(x_0)-x\cdot \nabla u(x_0)) K(x) \, dx,
\end{equation}
for any integer $1\leq k <n$, which arises from imposing certain geometric conditions on the kernels. Moreover, we will see that
$|y|^{-n-2s}\in \K_1^s \subset \K_k^s \subset \K_n^s,$ for $1<k<n$,
and thus, this family will be monotone decreasing, and bounded from above by the fractional Laplacian and by below by the Caffarelli--Silvestre nonlocal Monge-Amp\`{e}re.

The paper is organized as follows. In Section~\ref{sec:kernels}, we construct the family of admissible kernels $\K_k^s$, and give the precise definition of our operators for $C^{1,1}$-functions.
We introduce in Section~\ref{sec:RMP} the basic tools from the theory of rearrangements necessary for our goals. 
In Section~\ref{sec:inf}, we study the infimum in \eqref{eq:defop} and obtain a representation formula, provided some condition on the level sets is satisfied (see Theorem~\ref{thm:inf1}).  We also study the limit as $s\to1$ and give a connection to optimal transport.
 The H\"{o}lder continuity of $\F_k^s u$ is proved in Section~\ref{sec:reg}, following similar geometric techniques from \cite{CS}. 
In Section~\ref{sec:GPP}, we consider a global Poisson problem, prescribing data at infinity, and introduce a new definition of our operators for functions that are merely continuous and convex. We show existence of solutions via Perron's method and $C^{1,1}$-regularity in the full space by constructing appropriate barriers. 
Finally, we discuss some future directions in Section~\ref{sec:continuous}.

%%%%%%%%%%%%%%%%%%%%%%%%%%%%
\section{Construction of kernels} \label{sec:kernels}
%%%%%%%%%%%%%%%%%%%%%%%%%%%%

Let us start with the construction of the family of admissible kernels. Notice that any kernel $K$ in $\K_n^s$, defined in \eqref{eq:kernelsMA}, will have the same distribution function as the kernel of the fractional Laplacian, since for any $r>0$,
$$
\big\{x\in \Rn: |x|^{-n-2s} > r^{-n-2s} \big\} = B_r.
$$
Geometrically, this means that the level sets of $K$ are deformations in \textit{any} direction of $\Rn$ of the level sets of $|x|^{-n-2s}$, preserving the $n$-dimensional volume. 

In view of this approach, a natural way of finding an intermediate family of operators between the nonlocal Monge-Amp\`{e}re and the fractional Laplacian is to consider kernels whose level sets are deformations that preserve the $k$-dimensional Hausdorff measure $\H^k$, with $1\leq k<n$, of the restrictions of balls in $\Rn$ to hyperplanes generated by $\{e_i\}_{i=1}^k$.
\begin{figure}[h]
\centering
\begin{tikzpicture}[scale=0.6, use Hobby shortcut, closed=true]
 \shade [ball color = gray!40, opacity = 0.4] (0,0) circle [radius=3.03cm];
\draw [rotate around={0.:(0.,0.)},line width=0.4pt,dash pattern=on 1pt off 1pt] (0.,0.) ellipse (3.018521736515316cm and 0.4043185301409781cm);
\draw [rotate around={-0.5320261478892822:(0.018883213101843366,2.033541627541986)},line width=0.4pt,dash pattern=on 1pt off 1pt] (0.018883213101843366,2.033541627541986) ellipse (2.2110844521885764cm and 0.2766632996861545cm);
\draw [rotate around={-0.41249292988958997:(-0.014378375386446067,-1.9971396673078181)},line width=0.4pt,dash pattern=on 1pt off 1pt,fill=black,fill opacity=0.15] (-0.014378375386446067,-1.9971396673078181) ellipse (2.2504633222809285cm and 0.3245901313115684cm);
\draw [line width=0.4pt,dash pattern=on 1pt off 1pt] (-2.25,-2)-- (-3.1,-7.1);
\draw [line width=0.4pt,dash pattern=on 1pt off 1pt] (2.25,-2)-- (3.03,-6.85);
 \draw[line width=0.4pt,dash pattern=on 1pt off 1pt,fill=black,fill opacity=0.23]  (-3,-7)..(-2.5,-7.5) .. (-2,-6)..(-1.5,-5.5) .. (-1,-6.4).. (0,-6).. (0.5,-5.5).. (1,-5.3)..(1.5,-6).. (2,-6.2).. (3,-7)..(2.5,-7.5).. (2,-8.3).. (1,-7.5)..(0.5,-7.8).. (0,-8.1) .. (-1.5,-8.3) ..(-2.5,-8.5)..(-3,-7.5);
\begin{scriptsize}
\draw [fill=black] (0.,0.) circle (1.5pt);
\draw[color=black] (0.3,-7) node {$\{K(\cdot,z)>r^{-n-2s}\}$};
\draw[color=black] (6.7, -7) node {$\langle e_1, e_2 \rangle+z e_3$};
\draw[color=black] (4,2) node {$B_r \subset \mathbb{R}^3$};
\draw (6.,-5.)-- (3.,-9.);
\draw (-6,-9)-- (3,-9);
\draw (-3,-5)-- (-6.,-9.);
\draw (-3,-5)-- (6,-5);
\end{scriptsize}
\end{tikzpicture}
\caption{Area preserving deformation in $\R^3$.}
\label{fig:kernels}
\end{figure}
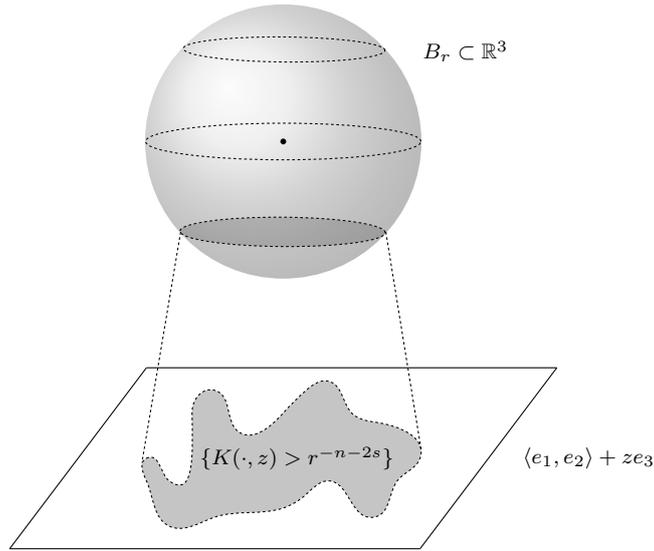

We define the set of admissible kernels as follows.  
\begin{defn} \label{def:kernels}
We say that $K \in \K_{k}^s$ if for all $z \in \R^{n-k}$, and all $r>0$,  it holds that
\begin{align} \label{eq:kernels}
\H^{k}\big(\{ y \in \R^k : K(y,z) > r^{-n-2s}\}\big) = 
\begin{cases}
\H^k\big(B_{(r^2-|z|^2)^{1/2}}\big) & \hbox{if}~|z|<r\\
0 & \hbox{if}~|z|\geq r,
\end{cases}
\end{align}
where $B_{(r^2-|z|^2)^{1/2}}$ is the ball in $\R^k$ of radius $(r^2-|z|^2)^{1/2}$. 
\end{defn}

In Figure~\ref{fig:kernels} we illustrate condition \eqref{eq:kernels} for $k=2$ and $n=3$.
Note that for $k=n$, we recover the definition of $\K_{n}^s$. Moreover, $|x|^{-n-2s} \in \K_{k}^s$, for all $k$.

\begin{prop} \label{prop:order}
Let $1\leq k<n$. Then $\K_{k}^s \subset \K_{k+1}^s \subseteq \K_{n}^s$.
\end{prop}

\begin{proof}
Let $K\in \K_{k}^s$. Fix any $z\in \R^{n-k-1}$ and $r>0$.
Then:
\begin{align*}
\H^{k+1}\big(\{ y \in \R^{k+1} : K(y,z) > r^{-n-2s}\}\big) 
&= \int_{\R^{k+1}} \chi_{\{ y \in \R^{k+1} : K(y,z) > r^{-n-2s}\}}(y) \, dy\\
&=\int_\R \Big( \int_{\R^k} \chi_{\{ (w,t) \in \R^{k}\times \R \, : \, K(w,t,z) > r^{-n-2s}\}}(w,t) \, dw \Big)dt\\
&=\int_\R  \H^k\big(\{ w \in \R^{k} : K(w,t,z) > r^{-n-2s}\}\big)\, dt \equiv {\rm I}.
\end{align*}
If $|z|\geq r$, then for any $t\in \R$, we have that $(t,z)\in \R^{n-k}$, with $|(t,z)|>r$. Therefore, by Definition~\ref{eq:kernels}, it follows that ${\rm I}=0$.
If $|z|<r$, then
\begin{align*} 
{\rm I} &= \int_\R \H^k\big(B_{(r^2-t^2-|z|^2)^{1/2}}\big)\, dt\\
&= \omega_k \int_{-(r^2-|z|^2)^{1/2}}^{(r^2-|z|^2)^{1/2}} (r^2-t^2-|z|^2)^{\frac{k}{2}}\, dt\\
&=\omega_k  (r^2-|z|^2)^{\frac{k}{2}} \int_{-(r^2-|z|^2)^{1/2}}^{(r^2-|z|^2)^{1/2}} \Big(1-\Big(\frac{t}{(r^2-|z|^2)^{1/2}}\Big)^2\Big)^{\frac{k}{2}}\, dt\\
&= \omega_k (r^2-|z|^2)^{\frac{k+1}{2}} \int_{-1}^1 (1-\sigma^2)^{k/2}\, d\sigma\\
&= \frac{\pi^{k/2}}{\Gamma\big(\frac{k}{2}+1\big)}\frac{\pi^{1/2} \Gamma\big(\frac{k}{2}+1\big)}{\Gamma\big(\frac{k+1}{2}+1\big)} (r^2-|z|^2)^{\frac{k+1}{2}}\\
&= \omega_{k+1} (r^2-|z|^2)^{\frac{k+1}{2}}= \H^{k+1}\big(B_{(r^2-|z|^2)^{1/2}}\big),
\end{align*}
where $\omega_l= \H^l(B_1)= \frac{\pi^{l/2}}{\Gamma\big(\frac{l}{2}+1\big)}$, and $B_{(r^2-|z|^2)^{1/2}}$ is the ball of radius $(r^2-|z|^2)^{1/2}$ in $\R^{k+1}$.
\end{proof}

\begin{defn} \label{def:c11}
A function $u : \Rn \to \R$ is said to be $C^{1,1}$ at the point $x_0$, and we write $u\in C^{1,1}(x_0)$, if there is a vector $p\in \Rn$, a radius $\rho>0$, and a constant $C>0$, such that
$$
|u(x_0+x)-u(x_0)-   x\cdot p| \leq C |x|^2, \quad \text{for all}~x\in B_\rho.
$$
We denote by $[u]_{C^{1,1}(x_0)}$, the minimum constant for which this property holds, among all admissible vectors $p$ and radii $\rho$.
\end{defn}

\begin{defn} \label{def:IO}
Let $s\in(1/2,1)$ and $1\leq k<n$. For any $u\in C^0(\Rn)\cap C^{1,1}(x_0)$, we define
$$
\F_k^su(x_0) =  c_{n,s}\inf_{K\in \K_{k}^s} \int_{\Rn}  \big(u(x_0+x)-u(x_0)-x\cdot \nabla u(x_0) \big) K(x)\, dx,
$$
where $\K_{k}^s$ is the set of kernels satisfying \eqref{eq:kernels} and $c_{n,s}$ is the constant in $\Delta^s$.
\end{defn}

As an immediate consequence of Proposition~\ref{prop:order}, we obtain that the operators are ordered.

\begin{cor} \label{cor:order}
Let $s\in(1/2,1)$ and $1\leq k<n$. Then for any $u \in C^0(\Rn)\cap  C^{1,1}(x_0)$,
$$
\MA^su(x_0) \leq \F_k^su(x_0)\leq \Delta^{s} u(x_0).
$$
Moreover, $\{\F_k^s \}_{k=1}^{n-1}$ is monotone decreasing.
\end{cor}

The regularity condition on $u$ in Definition~\ref{def:IO} allows us to compute $\F_k^su$ at the point $x_0$ in the classical sense. To obtain a finite number, we need to impose two extra conditions: 
\begin{enumerate}[$(P_1)$]
\item An integrability condition at infinity:
$$
\int_{\Rn} \frac{|u(x)|}{(1+|x|)^{n+2s}}\, dx<\infty.
$$
\item A convexity condition at $x_0$:
$$
\tilde u (x)\equiv u(x_0+x)-u(x_0)-x\cdot \nabla u(x_0) \geq 0, \quad\hbox{for all}~x\in \Rn. 
$$
\end{enumerate}

\begin{prop}\label{prop:finite}
If $u\in C^0(\Rn)\cap C^{1,1}(x_0)$ and satisfies $(P_1)$ and $(P_2)$, then
$$
0 \leq \F_k^s u(x_0) < \infty.
$$
\end{prop}
\begin{proof}
Let $\rho>0$ be as in Definition~\ref{def:c11}. Then
\begin{align*}
0\leq \F_k^su(x_0) &\leq \int_{\Rn} \big(u(x_0+x)-u(x_0)-x\cdot \nabla u(x_0) \big) \frac{1}{|x|^{n+2s}}\, dx\\
&\leq  \int_{B_\rho} \frac{[u]_{C^{1,1}(x_0)} |x|^2}{|x|^{n+2s}} \, dx + \int_{\Rn \setminus B_\rho(x_0)} \frac{|u(x)|}{|x-x_0|^{n+2s}}\, dx\\
&\qquad+|u(x_0)| \int_{\Rn \setminus B_\rho} \frac{1}{|x|^{n+2s}}\, dx+|\nabla u(x_0)|\int_{\Rn \setminus B_\rho} \frac{ |x|}{|x|^{n+2s}}\, dx \\
& \leq C(s,\rho) \big( |u(x_0)| + |\nabla u(x_0)| + [u]_{C^{1,1}(x_0)}\big)\\
&\qquad + \tfrac{1+|x_0|+\rho}{\rho}\int_{\Rn} \frac{|u(x)|}{(1+|x|)^{n+2s}}\, dx<\infty, \quad \hbox{since}~s\in (1/2,1).
\end{align*}
\end{proof}

We point out that if $u$ is not convex at $x_0$, then the infimum could be $-\infty$. We show this result in the next proposition.
\begin{prop}
Let $u\in C^0(\Rn)\cap  C^{1,1}(x_0)$. Assume that $u$ satisfies $(P_1)$.
 If there exists $\bar{x}\in \Rn$ with $\bar{x}=(\bar{y},0)$ and $\bar y\in \R^k$, such that
$$\tilde{u}(\bar{x}) = u(x_0+\bar{x})-u(x_0)-\bar{x} \cdot\nabla u(x_0)<0,$$ 
then $\F_k^s u (x_0)=-\infty$.
\end{prop}

\begin{proof}
Let $K(x)=|x-\bar{x}|^{-n-2s}.$ For any $r>0$ and $z\in \R^{n-k}$, if $|z|<r$, then
\begin{align*}
\H^{k}\big(\{ y \in \R^k : K(y,z) > r^{-n-2s}\}\big)&=
\H^{k}\big(\{ y \in \R^k : |y-\bar y|^2 +|z|^2 < r^{2}\}\big)= \H^k\big(B_{(r^2-|z|^2)^{1/2}}\big).
\end{align*}
Also, the measure is clearly zero if $|z|\geq r$.
Therefore, $K\in \K_{k}^s$. It follows that
\begin{align*}
\F_k^s u(x_0) & \leq \int_{\Rn} \tilde{u}(x) |x-\bar x|^{-n-2s}\, dx\\
& = \int_{B_\vep(\bar x)} \tilde{u}(x) |x-\bar x|^{-n-2s}\, dx + \int_{\Rn \setminus B_\vep(\bar x)} \tilde{u}(x) |x-\bar x|^{-n-2s}\, dx \equiv {\rm I+II}.
\end{align*}
Since $u\in C^0(\Rn)\cap C^{1,1}(x_0)$, we have that $\tilde{u}$ is continuous. Hence, given that $\tilde{u}(\bar x)<0$, then $\tilde{u}(x) <0$, for all $x \in B_\vep(\bar x)$, for some $\vep >0$. Moreover, since $K \notin L^1(B_\vep(\bar x))$, it follows that ${\rm I}=-\infty$. 
Arguing similarly as in the proof of Proposition~\ref{prop:finite}, we see that ${\rm II}<\infty$. Therefore, $$\F_k^s u (x_0)=-\infty.$$
\end{proof}

\begin{rem} \label{rem:rotation}
The operators $\F_k^s$ are not rotation invariant. This is because, for simplicity, 
in the construction of the family of admissible kernels $\K_{k}^s$ we chose the first $k$ vectors from the canonical basis of $\Rn$. In general, we may take any subset of $k$ unitary vectors, $\tau=\{\tau_i\}_{i=1}^k$, and replace the first condition on \eqref{eq:kernels} by  
\begin{equation} \label{eq:kernels2}
\H^{k}\big(\{ y \in \langle  \tau \rangle^\perp  : K(y+z\tau) > r^{-n-2s}\}\big) = \H^k\big(B_{(r^2-|z|^2)^{1/2}}\big),
\end{equation}
for all $z\in \langle \tau \rangle$ and $r>0$, where $\langle \tau \rangle$ denotes the span of $\{\tau_i\}_{i=1}^k$, and $\langle \tau \rangle^\perp$ the orthogonal subspace to $\langle \tau \rangle$.
Let $SO(n)$ be the group of rotation matrices $n\times n$. Since $\tau_i=A e_i$, for some $A\in SO(n)$, it follows that any kernel $K_\tau$ satisfying \eqref{eq:kernels2} can be written as $K_\tau = K \circ A$, where $K$ satisfies \eqref{eq:kernels}. 
Therefore to make the operators rotation invariant, one possibility is to take the infimum over all possible rotations. Namely,
$$
\inf_{A\in SO(n)}  \inf_{K\in \K_{k}^s} \int_{\Rn} \tilde u(x) K(Ax)\, dx.
$$
To focus on the main ideas, we will not explore this operator in this work.
\end{rem}

%%%%%%%%%%%%%%%%%%%%%%%%%%%%%%%%%%%%%%%%%%
\section{Rearrangements and measure preserving transformations} \label{sec:RMP}
%%%%%%%%%%%%%%%%%%%%%%%%%%%%%%%%%%%%%%%%%%

We introduce some definitions and preliminary results regarding rearrangements of nonnegative functions. For more detailed information, see for instance \cite{BS,B}.

\begin{defn} \label{def:rearrange}
Let $f : \Rn \to \R$ be a nonnegative measurable function. We define the decreasing rearrangement of $f$ as the function $f^*$ defined on $[0,\infty)$ given by
\begin{equation*} \label{eq:DR}
f^*(t) = \sup \big\{ \lambda > 0 : |\{x\in \Rn : f(x) > \lambda \}|>t \big\},
\end{equation*}
and the increasing rearrangement of $f$ as the function $f_*$ defined on $[0,\infty)$ given by
\begin{equation*} \label{eq:IR}
f_*(t) = \inf \big\{ \lambda > 0 : |\{x\in \Rn : f(x) \leq \lambda \}|>t \big\}.
\end{equation*}
We use the convention that $\inf \emptyset = \infty$. 
\end{defn}

\begin{prop} \label{prop:HLI}
Let $f,g : \Rn \to \R$ be nonnegative measurable functions.
Then
$$
\int_{0}^\infty  f_*(t) g^*(t)\, dt \leq \int_{\Rn} f(x) g(x)\, dx \leq \int_{0}^\infty  f^*(t) g^*(t)\, dt   .
$$
\end{prop}
The upper bound is the classical Hardy--Littlewood  inequality. For the proof see \cite[Theorem 2.2]{BS} or \cite[Corollary 2.16]{B}.
For the sake of completeness, we give the proof of the lower bound.
\begin{proof}
For $j\geq 1$, let $f_j= f|_{B_j}$ and $g_j= g|_{B_j}$, where $B_j$ denotes the ball of radius $j$ centered at $0$ in $\Rn$.
By \cite[Corollary~2.18]{B}, it follows that
$$
\int_0^{|B_j|} (f_j)_*(t) (g_j)^*(t)\, dt \leq \int_{B_j} f_j(x) g_j(x)\, dx.
$$ 
Since $f,g\geq0$, we get that
\begin{equation*} 
 \int_{B_j} f_j(x) g_j(x)\, dx \leq \int_{\Rn} f(x) g(x)\, dx.
\end{equation*}
 Note that for any $t \in [0, |B_j|]$, we have
$$\big\{\lambda>0 : |\{x\in B_j : f_j(x) \leq \lambda\}|>t\big\} \subset \big\{\lambda>0 : |\{x\in \Rn: f(x) \leq \lambda\}|>t\big\}.$$
Hence, $(f_j)_*(t) \geq f_*(t)$, and
\begin{equation*} 
\int_0^{|B_j|} (f_j)_*(t) (g_j)^*(t)\, dt \geq \int_0^{|B_j|} f_*(t) (g_j)^*(t)\, dt. 
 \end{equation*}
Moreover, $g_j \nearrow g$ pointwise on $\Rn$. Then by \cite[Proposition~1.39]{B}, we have 
$(g_j)^*\nearrow g^*$ pointwise on $[0, \infty)$, as $j\to \infty$. By the monotone convergence theorem, we get
\begin{equation*}
\lim_{j \to \infty}  \int_0^{|B_j|} f_*(t) (g_j)^*(t)\, dt =  \int_{0}^\infty  f_*(t) g^*(t)\, dt.
\end{equation*}
Combining the previous estimates, we conclude that
$$
\int_{0}^\infty  f_*(t) g^*(t)\, dt \leq \int_{\Rn} f(x) g(x)\, dx. 
$$
\end{proof}

\begin{defn}\label{def:MP}
We say that a measurable function $\psi : \R^l \to \R^m$ is a measure preserving transformation if for any measurable set $E$ in $\R^m$, it holds that
$$
\H^l(\psi^{-1}(E)) = \H^m(E).
$$
\end{defn}

\begin{lem} \label{lem:mp}
If $\psi : \R^l \to \R^m$ is a measure preserving, then for any measurable $f: \R^m\to \R$, and any measurable set $E$ in $\R^m$, it follows that
$$\int_{E} f(y) \, dy = \int_{\psi^{-1} (E)} f(\psi(z))\, dz.$$
\end{lem}

An important result by Ryff \cite{Ryff} provides a sufficient condition for which we can recover a function given its decreasing/increasing rearrangement, by means of a measure preserving transformation. 
\begin{thm}[Ryff's theorem] \label{thm:ryff}
Let $f: \Rn \to \R $ be a nonnegative measurable function. If 
$\lim_{t\to \infty} f^*(t)=0,$
then there exists a measure preserving $\sigma : \supp(f) \to \supp(f^*)$ such that
$$
f= f^* \circ \sigma
$$
almost everywhere on the support of $f$.
Similarly, if $\lim_{t\to \infty} f_*(t)=\infty$, then $f= f_* \circ \sigma$.
\end{thm}

 We will call  Ryff's map, a measure preserving $\sigma$ satisfying Ryff's theorem.
 \begin{rem}
 In general, $\sigma$ is not invertible. Furthermore, there may not exist a measure preserving transformation $\psi$ such that  $f^* = f \circ \psi$. 
 \end{rem}
 As a consequence of Ryff's theorem, we obtain a representation formula for the admissible  kernels. We denote  $\omega_k=\H^k(B_1)$.
 
 \begin{lem} \label{lem:repker}
Let $K \in \K_{k}^s$. Fix $z\in \R^{n-k}$ and denote by $K_z(y)=K(y,z)$. Then
$$
K_z^*(t) = \big(\big(\omega_k^{-1}t\big)^{2/k} + |z|^2 \big)^{-\frac{n+2s}{2}}.
$$
In particular, there exists a measure preserving $\sigma_z : \supp(K_z) \to (0,\infty)$, such that
 $$
 K(y,z) = K_z^*(\sigma_z(y)), \quad \hbox{for}~a.e.~y\in \supp(K_z).
 $$
 \end{lem}
 
 \begin{proof}
Fix  $z\in \R^{n-k}$.  Then
\begin{align*}
K_z^*(t) &= \sup \big\{ \lambda > 0 : \H^k\big(\{y\in \R^k : K(y,z) > \lambda \}\big)>t \big\}\\
&= \sup \big\{\lambda < |z|^{-n-2s} :   \H^k\big( B_{(\lambda^{-2/(n+2s)} - |z|^2)^{1/2}}\big) > t \big\}\\
&= \sup \big\{\lambda < |z|^{-n-2s} : \omega_k (\lambda ^{-2/(n+2s)}- |z|^2)^{k/2} > t \big\}\\
&= \sup \big\{ \lambda < |z|^{-n-2s} : \lambda ^{-2/(n+2s)}>  \big(\omega_k^{-1}t\big)^{2/k} + |z|^2 \big\}\\
&=\big(\big(\omega_k^{-1}t\big)^{2/k} + |z|^2 \big)^{-\frac{n+2s}{2}}.
\end{align*}
Moreover,
$
\lim_{t\to \infty} K_z^*(t)=0.
$
Therefore, the result  follows from Theorem~\ref{thm:ryff}.
 \end{proof}

In view of Definition~\ref{def:rearrange}, we introduce the symmetric rearrangement of a function in $\Rn$ with respect to the first $k$ variables as follows. 
Fix $k \in \N$ with $1\leq k < n$. 
Given $x\in \Rn$, we denote $x=(y,z)$, with $y\in \R^k$ and $z\in \R^{n-k}$. 
Furthermore, for $z$ fixed, we call $f_z$ the restriction of $f$ to $\R^k$. Namely, $f_z(y)=f(y,z)$.

\begin{defn}\label{def:rearrangement}
Let $f:\Rn \to \R$ be a nonnegative measurable function.
We define the $k$-symmetric \textit{decreasing} rearrangement of $f$ as the function $f^{*,k}: \Rn \to [0,\infty]$ given by
$$
f^{*,k}(x) = f_z^*(\omega_k |y|^k), 
$$
and the $k$-symmetric \textit{increasing} rearrangement as the function $f_{*,k} : \Rn \to [0,\infty]$  given by
$$
f_{*,k}(x) = (f_z)_*(\omega_k |y|^k).
$$
\end{defn}

When $k=n$, we obtain the usual symmetric rearrangement.

\begin{rem} \label{rem:rearrange}
(1) Notice that $f^{*,k}$ and $f_{*,k}$ are radially symmetric and monotone decreasing/increasing, with respect to $y$. In the literature, this type of symmetrization is also known as the Steiner symmetrization \cite[Chapter 6]{B}. \medskip

\noindent (2)
By Lemma~\ref{lem:repker}, we see that any kernel $K \in \K_{k}^s$ satisfies  
\begin{equation} \label{rem:ker}
K^{*,k}(x)=|x|^{-n-2s}, \qquad \hbox{for}~x\neq 0.
\end{equation}
\end{rem}

%%%%%%%%%%%%%%%%%%%%%%%%%%%%
\section{Analysis of $\F_k^s$} \label{sec:inf}
%%%%%%%%%%%%%%%%%%%%%%%%%%%%

Our main goal of this section is to study the infimum in the definition of the operator,
$$
\F_k^su(x_0) =  c_{n,s}\inf_{K\in \K_{k}^s} \int_{\Rn} \tilde u(x) K(x)\, dx,
$$ 
where $\tilde u(x) = u(x_0+x)-u(x_0)-x\cdot \nabla u(x_0)$. Throughout the section, we will assume that $u\in C^0(\Rn)\cap C^{1,1}(x_0)$ and satisfies properties $(P_1)$ and $(P_2)$, so that $0\leq \F_k^s u(x_0) <\infty$.\medskip

\subsection{Analysis of the infimum} We will study the following cases:\medskip

{\bf Case 1.} For all $\lambda> 0$ and $z\in \R^{n-k}$, 
$$\H^k \big ( \{y \in \R^k : \tilde{u}(y,z) \leq  \lambda \} \big) < \infty.$$
 
{\bf Case 2.} There exists some $\lambda_0> 0$ such that for all $z\in \R^{n-k}$, 
$$
\H^k \big ( \{y \in \R^k : \tilde{u}(y,z) \leq \lambda \} \big)
\begin{cases}
< \infty & \hbox{for} \ 0< \lambda < \lambda_0 \\
=\infty & \hbox{for} \ \lambda \geq \lambda_0.
\end{cases}
$$

\medskip

{\bf Case 3.}  For all $\lambda> 0$ and $z\in \R^{n-k}$, 
$$\H^k \big ( \{y \in \R^k : \tilde{u}(y,z) \leq \lambda \} \big) = \infty.$$

In the first case, when all of the level sets of $\tilde{u}$ have finite measure,  we show that the infimum is attained at some kernel whose level sets depend on the measure preserving transformation that rearranges the level sets of $\tilde u$. More precisely:

\begin{thm}\label{thm:inf1}
Suppose that for all $\lambda> 0$ and $z\in \R^{n-k}$,
$$\H^k \big ( \{y \in \R^k : \tilde{u}(y,z) \leq \lambda \} \big) < \infty.$$ 
Then, for any $z\in \R^{n-k}$, there exists a measure preserving $\sigma_z : \R^k \to [0,\infty)$ such that 
\begin{equation*}
\F_k^s u (x_0) = c_{n,s} \int_{\R^{n-k}} \int_{\R^k}\frac{\tilde{u}(y,z) }{\big((\omega_k^{-1}\sigma_z(y))^{2/k}+|z|^2\big)^{\frac{n+2s}{2}}} \, dydz.
\end{equation*}
In particular, the infimum is attained. 
\end{thm}

\begin{rem}
Observe that if $\tilde u(\cdot,z)$ is constant in some set of positive measure, then the kernel where the infimum is attained is not unique since the integral is invariant under any measure preserving rearrangement of $K$ within this set (see \cite{Ryff}).
\end{rem}

Before we give the proof of  Theorem~\ref{thm:inf1}, we need a lemma regarding the $k$-symmetric increasing rearrangement of $\tilde u$. By Definition~\ref{def:rearrangement}, this is given by the following expression:
\begin{equation*} \label{def:ksymm}
\tilde{u}_{*,k}(y,z) = \inf \big\{ \lambda > 0 : \H^k \big ( \{w \in \R^k : \tilde{u}(w,z) \leq \lambda \} \big) > \omega_k |y|^k\big\}.
\end{equation*}

\begin{lem} \label{lem:limitinfty}
Fix $z\in \R^{n-k}$. If $\H^k \big ( \{y \in \R^k : \tilde{u}(y,z) \leq \lambda \} \big) < \infty$, for all $\lambda> 0$, then 
\begin{equation*}
\lim_{|y| \to \infty} \tilde{u}_{*,k}(y,z) = \infty.
\end{equation*}
\end{lem}

\begin{proof}
 Assume there exists $M>0$, independent of $\lambda$, such that
\begin{equation}\label{eq:bddmeasure}
\H^k \big ( \{w \in \R^k : \tilde{u}(w,z) \leq \lambda \} \big)\leq M, \quad \hbox{for all}~\lambda>0. 
\end{equation}
Then for any $y \in \R^k$, with $\omega_k |y|^k>M$, we have that
$$
\tilde{u}_{*,k}(y,z)=\infty,
$$
since $\inf \emptyset=\infty$. 
If \eqref{eq:bddmeasure} does not hold, then there must be an increasing sequence $\{M_\lambda\}_{\lambda>0}$, with $M_\lambda \to \infty$, as $\lambda \to \infty$, such that 
$$
\H^k \big ( \{w \in \R^k : \tilde{u}(w,z) \leq \lambda \} \big)= M_\lambda. 
$$
Then for any $M>0$, there exists $\Lambda=\Lambda(M)>0$ such that $M_\lambda >M$, for all $\lambda >\Lambda$. Since $M_\lambda$ is monotone increasing, we can assume without loss of generality that $M_\Lambda \leq M$. Otherwise, we take  $\Lambda$ to be the minimum for which this property holds. Also, $\Lambda(M)$ is monotone increasing, and $\Lambda(M)\to \infty$, as $M\to \infty$. 
In particular, it holds that
$$
\inf \{\lambda > 0 : M_\lambda >M\} \geq  \Lambda(M) \to \infty \quad \hbox{as}~M\to \infty.
$$
Then for any $K>0$, there exists $M>0$ such that
$$
\inf \{\lambda > 0 : M_\lambda >M\}  \geq K.
$$
Therefore, for any $y\in \R^k$, with $\omega_k |y|^k > M$, we have
$$
\tilde{u}_{*,k}(y,z)= \inf \{\lambda > 0 : M_\lambda > \omega_k |y|^k\} 
\geq  \inf \{\lambda > 0 : M_\lambda >M\}  \geq  K.
$$
We conclude that
 $$\lim_{|y| \to \infty} \tilde{u}_{*,k}(y,z) = \infty.$$
\end{proof}

\begin{proof}[Proof of Theorem~\ref{thm:inf1}]
Since $u$ is convex at $x_0$, we have that $\tilde{u}(y,z)\geq 0$. Moreover, 
$$
\F_{k}^s u (x_0) = c_{n,s} \inf_{K\in \K_{k}^s} \int_{\R^{n-k}} \int_{\R^k} \tilde{u}(y,z) K(y,z) \, dydz.
$$
Fix  $z\in \R^{n-k}$ and consider the functions $f(y)=\tilde{u}(y,z)$ and $g(y)=K(y,z)$. Since 
$$
\H^k \big ( \{y \in \R^k : \tilde{u}(y,z) \leq \lambda \} \big) < \infty,
$$ 
for any $\lambda> 0$, then by Lemma~\ref{lem:limitinfty}, we have 
$$\lim_{t \to \infty} f_*(t)=\lim_{|y| \to \infty} f_{*,k}(x)=\infty,$$ 
with $f_{*,k}(x)=\tilde{u}_{*,k}(y,z)$ and $f_{*,k}(x)= f_*(\omega_k |y|^k)$.
By Ryff's theorem (Theorem~\ref{thm:ryff}), there exists a measure preserving $\sigma_z : \R^k \to [0,\infty)$, depending on $z$,  such that
\begin{equation}\label{eq:recoveru}
\tilde{u}(y,z)= f_*(\sigma_z(y)),
\end{equation}
for all $y\in \supp \tilde{u}(\cdot, z)\subseteq \R^k$.

Let
$K(y,z)=\big((\omega_k^{-1}\sigma_z(y))^{2/k}+|z|^2\big)^{-\frac{n+2s}{2}}.$ For any $r>|z|$, we have that
 \begin{align*}
 \H^k \big( \{ y \in \R^k : K(y,z)>r^{-n-2s}\} \big)
 &= \H^k \big( \{ y \in \R^k : \big((\omega_k^{-1}\sigma_z(y))^{2/k}+|z|^2\big)^{-\frac{n+2s}{2}}>r^{-n-2s}\} \big)\\
 &= \H^k \big( \{ y \in \R^k : \sigma_z(y)< \omega_k(r^2-|z|^2)^{k/2}\} \big)\\
&= \H^k \big(\sigma_z^{-1}\big((0,\omega_k (r^2-|z|^2)^{k/2})\big) \big)\\
 &=  \H^1\big(\big(0,\omega_k(r^2-|z|^2)^{k/2}\big)\big)\\
 &=\omega_k (r^2-|z|^2)^{k/2} = \H^k\big(B_{(r^2-|z|^2)^{k/2} }\big),
 \end{align*}
 since $\sigma_k$ is measure preserving (see Definition~\ref{def:MP}). Then $K\in\K_{k}^s$, and thus,
 $$
\F_{k}^s u (x_0) \leq c_{n,s}  \int_{\R^{n-k}} \int_{\R^k}\frac{\tilde{u}(y,z) }{\big((\omega_k^{-1}\sigma_z(y))^{2/k}+|z|^2\big)^{\frac{n+2s}{2}}} \, dydz.
$$
To prove the reverse inequality, let $K\in \K_{k}^s$. Applying Proposition~\ref{prop:HLI}, we see that
\begin{align*}
 \int_{\R^k} \tilde{u}(y,z) K(y,z) \, dy &\geq \int_0^\infty f_*(t) g^*(t)\, dt\\
 &=  \int_{\R^k} f_*(\sigma_z(y)) g^*(\sigma_z(y))\, dy\\
 &= \int_{\R^k} \tilde{u}(y,z) g^*(\sigma_z(y))\, dy,
\end{align*}
by Lemma~\ref{lem:mp} and \eqref{eq:recoveru}. Moreover, by the definition of rearrangements, 
\begin{align*}
g^*(\sigma_z(y)) 
&= \sup \big\{ \lambda> 0 : \H^k\big(\{w\in \R^k : K(w,z) > \lambda \}\big)>\sigma_z(y) \big\}= K^{*,k}(\tilde{y},z)
\end{align*}
with $\omega_k |\tilde{y}|^k=\sigma_z(y)$. By \eqref{rem:ker}, we get
$$
g^*(\sigma_z(y)) = \big(|\tilde{y}|^2+|z|^2\big)^{-\frac{n+2s}{2}}
=\big((\omega_k^{-1}\sigma_z(y))^{2/k}+|z|^2\big)^{-\frac{n+2s}{2}}.
$$
Hence, integrating over all $z\in \R^{n-k}$, and taking the infimum over all kernels $K\in \K_{k}^s$, we conclude that
$$
\F_k^su(x) = c_{n,s} \int_{\R^{n-k}} \int_{\R^k}\frac{\tilde{u}(y,z) }{\big((\omega_k^{-1}\sigma_z(y))^{2/k}+|z|^2\big)^{\frac{n+2s}{2}}} \, dydz.
$$
\end{proof}

\begin{rem}\label{rem:mp}
A natural question that arises from this result is whether there exists a measure preserving $\varphi_z : \R^k \to \R^k$ such that
$$
|\varphi_z(y)| = \big(\omega_k^{-1}\sigma_z(y)\big)^{1/k}.
$$
In that case, we would have that the infimum is attained at a kernel $K$ such that 
$$K(y,z)= |\phi(y,z)|^{-n-2s},$$
where $\phi : \Rn \to \Rn$ is a measure preserving with $\phi(y,z)=(\varphi_z(y),z)$. 

Recall that Ryff's theorem gives a representation of a function $f$ in terms of its increasing rearrangement $f_*$, that is, $f= f_* \circ \sigma$, with $\sigma : \R^k \to \R$ measure preserving. 
If this result were also true for the \textit{symmetric} increasing rearrangement, given by $f_\#(x) = f_*( \omega_k |x|^k)$, then there would exist a measure preserving $\varphi : \R^k \to \R^k$ such that $f= f_\# \circ \psi$. In particular, 
$$
f(x) =  f_\# (\varphi(x)) = f_*(\omega_k |\varphi(x)|^k)= f_*(\sigma(x)).
$$
Hence, it seems reasonable that $\omega_k |\varphi(x)|^k= \sigma(x)$.
As far as we know, this is an open problem.
\end{rem}

As an immediate consequence of Theorem~\ref{thm:inf1}, we obtain the following representation of the function $\F_{k}^s u$ in terms of the $k$-symmetric increasing rearrangement of $\tilde{u}$.
\begin{cor} \label{cor:inf}
Under the  assumptions of Theorem~\ref{thm:inf1}, we have
\begin{equation*}
\F_k^s u (x_0) = \Delta^s \tilde{u}_{*,k}(0).
\end{equation*}
\end{cor}

\begin{proof} 
Note that $\tilde{u}_{*,k}(0)=0$, since $\tilde{u}(0)=0$. Therefore, using the same notation as in the proof of Theorem~\ref{thm:inf1}, we showed that
\begin{align*}
\F_k^s u (x_0) &= c_{n,s} \int_{\R^{n-k}} \int_0^\infty f_*(t) g^*(t)\, dt dz\\
& = \omega_k  c_{n,s} \int_{\R^{n-k}} \int_0^\infty f_*(\omega_k r^k) g^*(\omega_k r^k) r^{k-1}\, drdz\\
 &=  c_{n,s} \int_{\R^{n-k}} \int_{\R^k} f_*(\omega_k |y|^k) g^* (\omega_k |y|^k ) \, dydz\\
 &=  c_{n,s} \int_{\R^{n-k}} \int_{\R^k} \tilde{u}_{*,k}(y,z) K^{*,k}(y,z)\, dydz\\
&=c_{n,s}  \int_{\R^{n-k}}\int_{\R^k} \frac{\tilde{u}_{*,k}(y,z)}{(|y|^2+|z|^2)^\frac{n+2s}{2}}\, dydz =  \Delta^s \tilde{u}_{*,k}(0).
\end{align*}
\end{proof}

From the previous result and the fact that the family of operators $\{\F_k\}_{k=1}^{n-1}$ is monotone decreasing, we see that the fractional Laplacian of the $k$-symmetric rearrangements are ordered at the origin.

\begin{cor} 
Suppose we are under the assumption of Theorem~\ref{thm:inf1}. Then
$$
   \Delta^s \tilde{u}_{*,k+1}(0) \leq \Delta^s \tilde{u}_{*,k}(0).
$$
\end{cor}

Next we treat the second case. 
\begin{thm}\label{thm:case2}
Suppose that there exists some $\lambda_0> 0$ such that for all $z\in \R^{n-k}$,
$$
\H^k \big ( \{y \in \R^k : \tilde{u}(y,z) \leq \lambda \} \big)
\begin{cases}
< \infty & \hbox{for} \ 0< \lambda < \lambda_0 \\
=\infty & \hbox{for} \ \lambda \geq \lambda_0.
\end{cases}
$$
Then there exists a kernel $K_0\in \K_{k}^s$, with $\supp K_0(\cdot,z) \subseteq \{y \in \R^k : \tilde{u}(y,z) \leq  \lambda_0 \}$, such that 
$$
\F_k^s u (x_0) = c_{n,s} \int_{\R^{n-k}} \int_{\R^k} \tilde{u}(y,z) K_0(y,z)\, dydz.
$$
In particular, the infimum is attained.
\end{thm}

\begin{proof}
Fix $z\in \R^{n-k}$. 
For $j\geq 1,$ define the set
\begin{align*}
 A_j(z) &= \big\{y \in \R^k : \tilde{u}(y,z) \leq \lambda_0 - \tfrac{1}{j} \big\}. 
\end{align*}
For simplicity, we drop the notation of $z$. We have that $\H^k(A_j) < \infty$, $A_j \subseteq A_{j+1}$, and 
$$A_\infty=\bigcup_{j=1}^\infty A_j =  \big\{y \in \R^k : \tilde{u}(y,z) < \lambda_0 \big\}.$$  

Observe that if $K\in \K_{k}^s$, then
\begin{align*}
\H^{k}\big(\{ y \in \R^k : K(y,z) > 0\}\big) = \lim_{r\to 0} \H^{k}\big(\{ y \in \R^k : K(y,z) > r\}\big)= \infty.
\end{align*}
Hence, we need to distinguish two cases:\medskip

\textbf{Case 2.1.} Assume that $\H^k(A_\infty) = \infty$. 
Let $K\in \K_{k}^s$ and $v_j = \tilde{u}\chi_{A_j}$. By Proposition~\ref{prop:HLI}, 
\begin{align*}
\int_{A_j} \tilde{u}(y,z) K(y,z)\, dy &= \int_{\R^k} v_j(y,z) K(y,z)\, dy \geq \int_0^\infty (v_j)_{*}(t) K^*(t) \, dt.
\end{align*}
By Lemma~\ref{lem:mp}, for any measure preserving $\sigma: \R^k \to [0,\infty)$, it follows that
$$
 \int_0^\infty (v_j)_{*}(t) K^*(t) \, dt =  \int_{\R^k} (v_j)_{*}(\sigma(y)) K^*(\sigma(y)) \, dy.
$$
By Ryff's theorem (Theorem~\ref{thm:ryff}), there exists $\sigma_j: A_j \to [0, \H^k(A_j)]$ measure preserving such that $v_j= (v_j)_{*} \circ \sigma_j$ in $A_j$. Therefore, 
\begin{equation} \label{eq:case1}
\int_{A_j} \tilde{u}(y,z) K(y,z)\, dy \geq \int_{A_j} \tilde{u}(y,z) K^*(\sigma_j(y)) \, dy.
\end{equation}
We claim that $\sigma_{j+1}(y) \leq \sigma_j(y)$, for all $y\in A_j$. Indeed, since $A_j \subseteq A_{j+1}$,  we have 
\begin{align*}
\begin{cases}
v_j(y) = v_{j+1}(y), & \hbox{for all}~y\in A_j \\
v_j(y) \leq v_{j+1}(y), & \hbox{for all}~y\in A_{j+1}\setminus A_j.
\end{cases}
\end{align*}
In particular, for all $y\in A_j$,
$$
 (v_{j+1})_*(\sigma_{j+1}(y)) = (v_j)_*(\sigma_j(y)) \leq (v_{j+1})_*(\sigma_j(y)). 
$$
Since $ (v_{j+1})_*$ is monotone increasing, we must have
$$\sigma_{j+1}(y) \leq \sigma_j(y), \quad\hbox{for all} ~y\in A_j.$$
Therefore, there exists $\sigma_\infty : A_\infty \to [0,\infty)$ measure preserving such that
$$
\sigma_\infty (y) = \lim_{j\to \infty} \sigma_j(y).
$$
Define the kernel $K_0$ as 
$$
K_0(y,z) =
\big( (\omega_k^{-1} \sigma_\infty (y))^{k/2} +|z|^2 \big)^{-\frac{n+2s}{2}}\chi_{A_\infty}(y).
$$
Since $\H^k(A_\infty)=\infty$, then $K_0\in \K_{k}^s$.
Furthermore, note that $\supp K_0( \cdot, z) = \overline{A_\infty}=\{y \in \R^k : \tilde{u}(y,z) \leq  \lambda_0 \}$ and
$K_0(y,z)=K_0^*(\sigma_\infty(y))$, for all $y\in A_\infty$.  Then by Fatou's lemma, Lemma~\ref{lem:repker}, and \eqref{eq:case1}, we get
\begin{align*}
\int_{\R^k} \tilde{u}(y,z) K_0(y,z)\, dy & =  \int_{A_\infty } \tilde{u}(y,z) K_0^*(\sigma_\infty(y))\, dy\\
&\leq \liminf_{j\to \infty} \int_{A_j } \tilde{u}(y,z) K_0^*(\sigma_j (y))\, dy\\
& = \liminf_{j\to \infty} \int_{A_j } \tilde{u}(y,z) K^*(\sigma_j (y))\, dy\\
& \leq \int_{\R^k}  \tilde{u}(y,z) K(y,z)\, dy,
\end{align*}
for any $K \in \K_{k}^s$. Integrating over $z$ and taking the infimum over all kernels $K$, we conclude the result.

\medskip
\textbf{Case 2.2.} Assume that $\H^k(A_\infty )< \infty$. Set $A=  \{y \in \R^k : \tilde{u}(y,z) =\lambda_0\}$. Then
\begin{equation} \label{eq:inftymeasure}
\H^k (A)= \infty,
\end{equation}
since 
$\{ y \in \R^k : \tilde u (y,z) \leq \lambda_0\} = A_\infty \cup A$.
Fix $\vep>0$ and define
\begin{align*}
 v_\vep(y,z) &= \tilde{u}(y,z)\chi_{A_\infty}(y) + \max\{ \lambda_0, (\lambda_0+\vep) \phi(y,z)\}\chi_{A}(y),
\end{align*}
with $\phi(y,z)=1- e^{-|y|^2-|z|^2}$.
Note that $0 < \phi\leq 1$, $\phi(y,z)\to 1$, as $|(y,z)|\to \infty$, and $\phi(y,z)\approx |y|^2+|z|^2$, as $|(y,z)|\to 0$. 
Also, $\{v_\vep\}_{\vep>0}$ is a monotone increasing sequence, and
\begin{align} \label{eq:limv}
\lim_{\vep\to 0} v_\vep(y,z) &= \tilde{u}(y,z)\chi_{A_\infty}(y) + \max\big\{ \lambda_0, \lim_{\vep\to 0}(\lambda_0+\vep) \phi(y,z)\big\}\chi_{A}(y)\\\nonumber
&= \tilde{u}(y,z)\chi_{A_\infty}(y) + \max\{ \lambda_0 , \lambda_0 \phi(y,z)\}\chi_{A}(y) 
=\tilde{u}(y,z) \chi_{A_\infty \cup A}(y). \nonumber
\end{align}

For any $j \in \N$, with $j > 1/\vep$, consider the set
$$
B_j^\vep (z) = \big\{ y \in \R^k : v_\vep(y,z) \leq \lambda_0+\vep -\tfrac{1}{j} \big\}.
$$
Then $B_{j}^\vep \subseteq B_{j+1}^\vep$ and $B_\infty^\vep= \bigcup_{j>1/\vep} B_j^\vep  = \{ y \in \R^k :  v_\vep(y,z) < \lambda_0+\vep \}$. 
Moreover, we have
\begin{align} \label{eq:measureBj}
\H^k(B_j^\vep)  & \leq \H^k(A_\infty) + \H^k \big(\big\{ y \in A : \max\{ \lambda_0, (\lambda_0+\vep) \phi(y,z)\} \leq \lambda_0+\vep -\tfrac{1}{j} \big\}\big).
\end{align}
Choose $R>0$ large enough (depending on $\vep$, $j$, $\lambda_0$, and $z$) so that 
$$
(\lambda_0+\vep) e^{-R^2-|z|^2} <  \tfrac{1}{j}.
$$
Then $(\lambda_0 + \vep) \phi(y,z) > \lambda_0 +\vep - \tfrac{1}{j}>\lambda_0$, for all $y \in B_R^c$, and thus,
\begin{align}\label{eq:meascompact}
\H^k \big(\big\{ y \in A \cap B_R^c  : \max\{ \lambda_0, (\lambda_0+\vep) \phi(y,z)\} \leq \lambda_0+\vep -\tfrac{1}{j} \big\}\big) = 0.
\end{align}
By \eqref{eq:measureBj} and \eqref{eq:meascompact}, we see that
$$
\H^k (B_j^\vep (z)) \leq \H^k(A_\infty) + \H^k(A\cap B_R) < \infty.
$$
Furthermore, $A \subseteq B_\infty^\vep$, and thus, by \eqref{eq:inftymeasure}, we get
\begin{align*}
\H^k(B_\infty^\vep)
&\geq   \H^k\big( A) = \infty.
\end{align*}
In particular, $v_\vep$ satisfies the assumptions of Case 2.1, so there exists $K_\vep \in \K_{k}^s$,
\begin{equation} \label{eq:kerapprox}
K_\vep(y,z) =\big( (\omega_k^{-1} \sigma_\vep (y))^{k/2} +|z|^2 \big)^{-\frac{n+2s}{2}}\chi_{B_\infty^\vep}(y),
\end{equation}
with $\sigma_\vep : B_\infty^\vep \to [0,\infty)$ measure preserving, depending on $v_\vep$, 
such that
\begin{equation} \label{eq:infapprox}
\inf_{K\in \K_{k}^s} \int_{\R^{n-k}} \int_{\R^k} v_\vep (y,z) K(y,z)\, dydz = \int_{\R^{n-k}} \int_{\R^k} v_\vep (y,z) K_\vep(y,z)\, dydz.
\end{equation}

Finally, we need to pass to the limit. First, we prove that  $\{\sigma_\vep\}_{\vep>0}$ is monotone decreasing. Indeed, let 
$
V_\vep = \{ y \in \R^k : v_\vep(y,z)=\tilde u(y,z)\}.
$
In particular, $A_\infty \subseteq V_\vep \subseteq A_\infty \cup A.$
Also, $V_{\vep_2}\subseteq V_{\vep_1}$, for any $\vep_1 \leq \vep_2$. By Ryff's theorem, recall that
\begin{align*}
v_{\vep_1}(y,z)= (v_{\vep_1})_* (\sigma_{\vep_1} (y)) \quad\hbox{and}\quad
v_{\vep_2}(y,z)= (v_{\vep_2})_* (\sigma_{\vep_2} (y)).
\end{align*}
Since $v_{\vep_2}(y,z)=v_{\vep_1}(y,z)$, for all $y\in V_{\vep_2}$, and $v_{\vep_1}(y,z)\leq v_{\vep_2}(y,z)$, for all $y\in \R^k$, we see that
\begin{align*}
(v_{\vep_2})_* (\sigma_{\vep_2} (y))=(v_{\vep_1})_* (\sigma_{\vep_1} (y)) \leq (v_{\vep_2})_* (\sigma_{\vep_1} (y)),
\quad \hbox{for all}~y\in V_{\vep_2}.
\end{align*}
Since $(v_{\vep_2})_*$ is monotone increasing, we must have that $\sigma_{\vep_2} (y)\leq \sigma_{\vep_1} (y)$, for all $y\in V_{\vep_2}$.
Hence, there exists $\sigma_0 : B_\infty \to[0,\infty) $ measure preserving such that
\begin{align*}
\sigma_0(y)=\lim_{\vep\to0} \sigma_\vep (y) ,
\end{align*}
where $B_\infty =\bigcap_{\vep>0} B_\infty^\vep = \{ y \in \R^k : \tilde u(y,z) \leq \lambda_0\}=A_\infty \cup A$. In particular, the sequence of kernels $\{K_\vep\}_{\vep>0}$ is monotone decreasing.
Define 
\begin{equation} \label{eq:kerlim}
K_0(y,z) = \lim_{\vep\to0} K_\vep(y,z).
\end{equation}
By \eqref{eq:kerapprox} and \eqref{eq:kerlim}, we have
$$
K_0(y,z)= \big( (\omega_k^{-1} \sigma_0 (y))^{k/2} +|z|^2 \big)^{-\frac{n+2s}{2}}\chi_{B_\infty}(y).
$$
Moreover, $K_0 \in \K_{k}^s$ since $K_\vep \in \K_{k}^s$, and for any $r>0$, it follows that
 $$
 \H^{k}\big(D_0(r)) = \lim_{\vep \to 0} \H^k(D_\vep(r)),
$$
where  $D_\vep(r) = \{ y \in \R^k : K_\vep (y,z) > r^{-(n+2s)}\}$.

 Finally, using \eqref{eq:limv}, \eqref{eq:infapprox}, \eqref{eq:kerlim}, and the monotone convergence theorem, we get
\begin{align*}
 \int_{\R^{n-k}}\int_{\R^k} \tilde u(y,z) K_0(y,z)\, dydz 
 & =  \int_{\R^{n-k}}\int_{\R^k} \lim_{\vep\to 0} \big( v_\vep (y,z) K_\vep(y,z)\big) \, dydz \\
 & = \lim_{\vep\to 0} \int_{\R^{n-k}}\int_{\R^k}  v_\vep (y,z) K_\vep(y,z) \, dydz \\
&= \lim_{\vep\to 0}\inf_{K\in \K_{k}^s} \int_{\R^{n-k}} \int_{\R^k} v_\vep (y,z) K(y,z)\, dydz \\
&\leq \inf_{K\in \K_{k}^s}  \int_{\R^{n-k}} \int_{\R^k} \Big( \lim_{\vep\to 0} v_\vep (y,z) \Big) K(y,z)\, dydz \\
&= \inf_{K\in \K_{k}^s}  \int_{\R^{n-k}} \int_{\R^k} \tilde u(y,z) \big( K(y,z) \chi_{A\infty \cup A}(y) \big)\, dydz\\
&= \inf_{K\in \K_{k}^s} \int_{\R^{n-k}}\int_{\R^k} \tilde u(y,z)  K(y,z) \, dydz.
\end{align*}
The last equality follows from the following observation: since 
$$
\tilde\K_{k}^s =\big \{ K \in \K_{k}^s : \supp K(\cdot,z) \subseteq A_\infty\cup A \big\} \subseteq  \K_{k}^s,
$$
then the infimum over all kernels in $ \K_{k}^s$ is less than or equal to the infimum over $\tilde  \K_{k}^s$.
Moreover, the reverse inequality holds trivially.
\end{proof}

Finally, we deal with the third case, that is, when all of the level sets of $\tilde u$ have infinite measure. 
In particular, notice that 
$$
\tilde{u}_{*,k}(x)=0, \quad \hbox{for all}~x\in \R^n.
$$   
This is the only case where the infimum is not attained. 
Indeed, we prove in the following theorem  that the infimum is equal to zero.

\begin{thm}
Suppose that for all $\lambda >0$ and $z\in \R^{n-k}$,
$$\H^k \big ( \{y \in \R^k : \tilde{u}(y,z) \leq \lambda \} \big) = \infty.$$
Then $\F_k^s u(x_0)=0$.
\end{thm}

\begin{proof}
From $(P_2)$, we have that $\F_k^s u(x_0)\geq 0$. To prove the reverse inequality, it is enough to find a sequence of  kernels $\{K_\vep\}_{\vep>0} \subset \K_{k}^s$ such that
\begin{equation} \label{eq:limit0}
\liminf_{\vep\to0}\int_{\R^{n-k}} \int_{\R^k}\tilde{u}(y,z) K_\vep(y,z) \, dydz = 0.
\end{equation}

Fix $\vep>0$ and $z\in \R^{n-k}$. For any $j\geq 0$, we define the set
\begin{align*}
U_j \equiv U_j (z) = \big\{ y \in \R^k : \tilde{u}(y,z) < \vep 2^{-j(n+2s)} e^{-|z|^2} \big\}.
\end{align*}
Note that $U_{j+1} \subseteq U_{j}$. Also, by assumption,  with $\lambda= \vep 2^{-j(1+2s)} e^{-|z|^2}$, we have that 
$$
\H^k(U_j) = \infty, \quad \hbox{for all}~ j\geq0.
$$ 
We will construct $K_\vep \in \K_{k}^s$ by describing first where to locate each level set of the form:
\begin{align*}
A_{-1} \equiv A_{-1}(z) &= \big\{ y \in \R^k : 0< K_\vep (y,z) \leq 1\big\}\\
A_j\equiv A_j(z) &= \big\{ y\in \R^k :  2^{j(n+2s)} < K_\vep (y,z) \leq 2^{(j+1)(n+2s)}\big\}, \quad \hbox{for}~j\geq0.
\end{align*}
Recall that $K\in \K_{k}^s $ if for all $r>0$, we have 
\begin{align*} \label{eq:kernels3}
\H^{k}\big(\{ y \in \R^k : K(y,z) > r^{-(n+2s)}\}\big) = \H^{k}\big(\{ y \in \R^k : (|y|^2+|z|^2)^{-\frac{n+2s}{2}} > r^{-(n+2s)}\}\big).
\end{align*}
In view of this definition, we define the sets
\begin{align*}
B_{-1} \equiv B_{-1}(z) &= \big\{ y \in \R^k : 0< (|y|^2+|z|^2)^{-\frac{n+2s}{2}}  \leq 1\big\}\\
B_j \equiv B_j(z) &= \big\{ y\in \R^k :  2^{j(n+2s)} < (|y|^2+|z|^2)^{-\frac{n+2s}{2}}  \leq 2^{(j+1)(n+2s)}\big\}, \quad \hbox{for}~j\geq0.
\end{align*}
Note that 
\begin{equation*}
\begin{cases}
\H^k(A_{-1}) = \H^k(B_{-1}) = \infty &\\
\H^k(A_j)   = \H^k(B_j)<\infty, & \hbox{for all}~j\geq0.
\end{cases}
\end{equation*}
More precisely, for $j\geq 0$, if $|z| < 2^{-(j+1)}< 2^{-j}$, then 
\begin{align*}
\H^k(A_j) & = \H^k\big(B_{(2^{-2j}-|z|^2)^{1/2}}\big) - \H^k\big(B_{(2^{-2(j+1)}-|z|^2)^{1/2}}\big)\\
&= \omega_k (2^{-2j}-|z|^2)^{k/2}- \omega_k (2^{-2(j+1)}-|z|^2)^{k/2}\leq \omega_k 2^{-kj}.
\end{align*}
If $2^{-(j+1)} \leq |z| < 2^{-j}$, then
\begin{align*}
\H^k(A_j) & = \H^k\big(B_{(2^{-2j}-|z|^2)^{1/2}}\big) =  \omega_k (2^{-2j}-|z|^2)^{k/2} \leq \omega_k (\tfrac{3}{4})^{k/2} 2^{-kj}.
\end{align*}
If $|z| \geq 2^{-j} > 2^{-(j+1)}$, then 
\begin{equation*} 
\H^k(A_j)=0.
\end{equation*}
Therefore, $\H^k(A_j)\leq c 2^{-kj}$, where $c>0$ only depends on $k$. It follows that
\begin{equation} \label{eq:finunion}
\H^k \Big( \bigcup_{j=0}^\infty A_j \Big) = \sum_{j=0}^\infty \H^k(A_j) \leq c \sum_{j=0}^\infty 2^{-jk} < \infty.
\end{equation}

For any $i \geq 0$, let $\D_i$ be the collection of all dyadic closed cubes of the form
$$
[m 2^{-i}, (m+1) 2^{-i}] ^k = [m 2^{-i}, (m+1) 2^{-i}] \times \cdots \times [m 2^{-i}, (m+1) 2^{-i}].
$$
Note that if $Q \in \D_i$, then $l(Q)=2^{-i}$, where $l(Q)$ denotes the side length of the cube $Q$.
For any $j\geq 0$, since $U_j$ is an open set, by a standard covering argument, we have that there exists a family of dyadic cubes $\F_j$ such that
$$
U_j = \bigcup_{Q\in \F_j} Q
$$
satisfying the following properties:
\begin{enumerate}
\item For any $Q\in \F_j$, there exists some $i\geq0$ such that $Q\in \D_i$.
\item $\Int(Q) \cap \Int(\tilde Q) = \emptyset$, for any  $ Q , \tilde Q \in \F_j$, with $Q \neq \tilde Q$.
\item If $x \in Q \in \F_j$, then $Q$ is the maximal dyadic cube contained in $U_j$ that contains $x$.
\end{enumerate} 
Analogously, for the sets $B_j$, with $j\geq -1$, there exists a family of dyadic cubes $\tilde \F_j$ satisfying properties $(1)-(3)$ such that
$$
\Int(B_j) = \bigcup_{Q\in \tilde \F_j} Q.
$$
Note that $\tilde\F_j \cap \tilde\F_{j+1} =\emptyset$ since $B_j \cap B_{j+1} = \emptyset$. 

We will construct the sets $A_j$ by properly translating the dyadic cubes partitioning the sets $B_j$ into $U_j$. 
In particular, we will  prove that 
$$
\begin{cases}
A_0 = T_0(B_0) \subset U_0\\
A_j = T_j(B_j)  \subset U_j \setminus \bigcup_{i=0}^{j-1} A_i, \quad \hbox{for all}~ j\geq 1\\
A_{-1} = T_{-1}(B_{-1}) \subset U_0  \setminus \bigcup_{i=0}^{\infty} A_i,
\end{cases}
$$
for some \textit{translation} mappings $T_j : \tilde \F_j  \to \F_j $ to be determined.

We start with the case $j=0$. For any $i\geq 0$, denote by 
\begin{align*}
m_i &=\H^0(\F_0 \cap \D_i) \quad \hbox{and} \quad n_i = \H^0( \tilde \F_0 \cap \D_i),
\end{align*}
where $\H^0(E)$ is equal to the cardinal of the set $E$. Note that $m_i, n_i \in \Z^+ \cup \{\infty\} $.

We will recursively place $B_0$ into $U_0$.
First, fix $i=0$. If $m_0 \geq n_0$, then for any $\tilde Q \in \tilde \F_0 \cap \D_0$, there exists some $\tau \in \R^k$ and some $Q\in
\F_0 \cap \D_0$, such that $Q = \tilde Q + \tau$. Then define 
\begin{equation} \label{eq:T0}
\begin{array}{rrcl} 
T_0 : & \tilde \F_0 \cap \D_0  & \to & \F_0 \cap \D_0\\
  & \tilde Q & \mapsto & Q.
\end{array}
\end{equation}
Moreover, we can define $T_0$  one-to-one since $m_0 \geq n_0$, and we can always choose a different $Q$ for each $ \tilde Q$. Note that there are $p_0$ cubes in $\F_0 \cap D_0$, with $p_0=m_0-n_0$, that have not been used. Hence, to all of these cubes,   
divide each side in half, so that each cube gives rise to $2^k$ cubes with side length $2^{-1}$. Call this collection of new cubes $\mathcal{Q} = \{Q_l\}_{l=1}^{2^{kp_0}} \subset \D_1$, and  add them to the family $\F_0 \cap \D_1$. Namely, we replace $ \F_0\cap \D_1$ by $(\F_0\cap \D_1) \cup \mathcal{Q}$.

If $m_0 < n_0$, then take $q_0$ cubes in $\tilde{\F}_0 \cap \D_0$, with  $q_0 = n_0 - m_0$, and divide each side in half. Call this collection of new cubes $\tilde{\mathcal{Q}} = \{\tilde Q_l\}_{l=1}^{2^{kq_0}} \subset \D_1$. 
Then, we replace $\tilde \F_0$ by $\hat{\F}_0$, where
\begin{align*}
\hat{\F}_0 \cap \D_0 & = ( \tilde \F_0 \setminus  \tilde{\mathcal{Q}}) \cap \D_0\\
\hat{\F}_0 \cap \D_1 & = ( \tilde \F_0 \cup \tilde{\mathcal{Q}}) \cap \D_1\\
\hat{\F}_0 \cap \D_i & = \tilde {\F}_0 \cap \D_i, \quad \hbox{for all}~ i \geq 2.
\end{align*}
If $\hat n_0 = \H^0( \hat \F_0 \cap \D_0)$, then  $m_0 = \hat n_0$. Hence, by the same argument as in the previous case, we find $T_0$ as in \eqref{eq:T0}. 
For $i\geq 1$, we can repeat the same process until we run out of cubes from $\tilde \F_0$ (or the modified family). We know the process will end since $\H^k(B_0) <\H^k(U_0)$.
When this happens, we will have constructed a one-to-one mapping $T_0 : \tilde \F_0 \to \F_0$, since $\tilde \F_0 = \bigcup_{i=0}^\infty \tilde\F_0\cap \D_i$ and $ \F_0 = \bigcup_{i=0}^\infty \F_0\cap \D_i$. Then define
$$
A_0 = T_0 (B_0) \subset U_0.
$$

Iterating this process, we find a sequence of translation mappings $\{T_j\}_{j= 0}^\infty$, with $T_j : \tilde F_j \to \F_j$, and a sequence of disjoint sets $\{A_j\}_{j=0}^\infty$ such that
$$
A_j = T_j(B_j)  \subset U_j \setminus {\textstyle \bigcup \limits_{i=0}^{j-1}} A_i.
$$

The case $j=-1$ is somewhat special since $\H^k(A_{-1})=\H^k(B_{-1})=\infty$.  
We will see that
$$
A_{-1} = T_{-1}(B_{-1}) \subset U_0  \setminus {\textstyle \bigcup \limits_{i=0}^{\infty}} A_i.
$$
This is possible because $\H^k(U_0  \setminus \bigcup_{i=0}^{\infty} A_i)=\infty$ using \eqref{eq:finunion}. Indeed, we can write
$$
\big\{ y \in \R^k : 0<K_\vep(y,z) \leq 1 \big\} = \bigcup_{j=0}^\infty \big\{ 2^{-(j+1)(n+2s)} < K_\vep(y,z) \leq 2^{-j(n+2s)}\big \}.
$$
Now call 
$$
C_j = \big\{ 2^{-(j+1)(n+2s)} < (|y|^2+|z|^2)^{-\frac{n+2s}{2}}  \leq 2^{-j(n+2s)} \big\}, \quad \hbox{for}~j\geq 0.
$$
Then  $B_{-1} =  \bigcup_{j=0}^\infty C_j$, with $\H^k(C_j)< \infty$, for all $j\geq 0$.
Hence, instead of partitioning all of $B_{-1}$ into dyadic cubes, we partition each of its disjoint components $C_j$. Arguing as before, we place them into  $U_0  \setminus \bigcup_{i=0}^{\infty} A_i$  recursively, according to the following scheme:
\begin{equation*}
\begin{cases}
T_{-1}^0(C_0) \subset U_0 \setminus {\textstyle \bigcup \limits_{i=0}^{\infty}} A_i & \\
T_{-1}^j (C_j) \subset U_0 \setminus \Big({\textstyle \bigcup \limits_{i=0}^{\infty}} A_i \cup  {\textstyle \bigcup \limits_{i=0}^{j-1}} C_i\Big), & \hbox{for}~j\geq 1,\\
\end{cases}
\end{equation*}
where $T_{-1}^j$ is defined as before.
 At the end of this process, we find a translation map $T_{-1}$ with $T_{-1}(Q) =  T_{-1}^j(Q)$, for $Q\in C_j$. Therefore, we define
 $$
 A_{-1} = T_{-1}(B_{-1}).
 $$

Lastly, let $y\in \R^k = A_{-1} \cup \big( \bigcup_{j=0}^\infty A_j \big)$. In particular, there exists some $j\geq{-1}$ such that $y \in A_j$. Furthermore, recall that $A_j = T_j(B_j)$, where $T_j$ is a one-to-one and onto translation map. Hence, there exists a unique $w\in B_j$  such that $y= T_j(w)=w+\tau$, for some $\tau \in \R^k$. Let $T_z: \R^k \to \R^k$ be given by $T_z(y)=w$. Note that $T_z$ is measure preserving. Then we define the kernel $K_\vep$ as
 $$
 K_{\vep}(y,z) = \big(|T_z(y)|^2 +|z|^2\big)^{-\frac{n+2s}{2}}. 
 $$

\medskip
We have that
\begin{align*}
\int_{\R^k}\tilde{u}(y,z) K_\vep(y,z) \, dy 
&= \int_{A_{-1}}\tilde{u}(y,z) K_\vep(y,z) \, dy + \sum_{j=0}^\infty \int_{A_j}\tilde{u}(y,z) K_\vep(y,z) \, dy \equiv {\rm I+II}.
\end{align*}
For I, we use that $ \tilde{u}(y,z)\leq \vep e^{-|z|^2}$, since $A_{-1} \subset U_0$. Then by Lemma~\ref{lem:repker} and Lemma~\ref{lem:mp}:
\begin{align*}
{\rm I} &\leq \vep e^{-|z|^2} \int_{\{  0<K_\vep(y,z)\leq 1\}} K_\vep(y,z)\, dy\\ 
&= \vep e^{-|z|^2} \int_{\{ 0<|\sigma_z(y)|^{-n-2s} \leq 1\}} |\sigma_z(y)|^{-n-2s}\, dy \\
& = \vep e^{-|z|^2} \int_{\{ |y|\geq 1\}} |y|^{-n-2s}\, dy = C\vep e^{-|z|^2},  
\end{align*}
where $C>0$ depends only on $n$ and $s$.
For II, we use that $ \tilde{u}(y,z)\leq \vep 2^{-j(n+2s)} e^{-|z|^2}$, since $A_j \subset U_j$ and $K_\vep (y,z) \leq 2^{(j+1)(n+2s)}$ in $A_j$, by definition. Then 
\begin{align*}
{\rm II}  &\leq   \vep e^{-|z|^2} \sum_{j=0}^\infty  2^{-j(n+2s)}  2^{(j+1)(n+2s)} \H^k(A_j)\\
& \leq c\vep e^{-|z|^2} 2^{n+2s} \sum_{j=0}^\infty  2^{-kj} \leq C \vep  e^{-|z|^2},
\end{align*}
where $C>0$ depends only on $n$, $s$, and $k$.

Integrating over $z$, we see that
\begin{align*}
\int_{\R^{n-k}} \int_{\R^k}\tilde{u}(y,z) K_\vep(y,z) \, dydz  
& \leq C \vep \int_{\R^{n-k}} e^{-|z|^2}\, dz \leq \tilde C\vep.
\end{align*}
Letting $\vep\to0$, we conclude \eqref{eq:limit0}.
\end{proof}

%%%%%%%%%%%%%%%%%%%%%%%%%%%%%%%%
\subsection{Limit as $s\to 1$} \label{sec:limit}
%%%%%%%%%%%%%%%%%%%%%%%%%%%%%%%%

Let $u\in C^2(\Rn)$.
We define $\MA_k u$ as the Monge-Amp\`{e}re operator acting on $u$, with respect to the first $k$ variables, that is,
\begin{equation*}\label{eq:kMA}
\MA_k u(x) = k\Big( \det\big((u_{ij}(x))_{1\leq i,j\leq k}\big) \Big)^{1/k},
\end{equation*}
with $D^2 u(x)=(u_{ij}(x))_{1\leq i,j\leq n}$. We define $\Delta_{n-k}u $ as the Laplacian of $u$, with respect to the last $n-k$ variables, that is,
\begin{equation*}\label{eq:kL}
\Delta_{n-k} u(x) = \sum_{i=k+1}^n  u_{ii}(x).
\end{equation*}
Then under some \textit{special} conditions, it holds that 
\begin{equation} \label{eq:limit}
\lim_{s\to1} \F_k^s u (x) =\MA_k u(x) + \Delta_{n-k} u(x). 
\end{equation}
In particular, the family $\{\F_k^s\}_{k=1}^{n-1}$ can be understood as nonlocal analogs of concave second order elliptic operators, which are decomposed into a Monge-Amp\`{e}re operator restricted to $\R^k$ and a Laplacian restricted to $\R^{n-k}$. 

Indeed, by Corollary~\ref{cor:inf}, we have $\F_k^s u (x) =  \Delta^s \tilde{u}_{*,k}(0)$.
Since the $k$-symmetric rearrangement does not depend on $s$ and $\Delta^s \to \Delta$, as $s\to 1,$ passing to the limit we see that
$$
\lim_{s\to1} \F_k^s u (x) =\Delta \tilde{u}_{*,k}(0).
$$
Suppose that
$
\tilde{u}_{*,k}(y,z)= \tilde u(\varphi_z^{-1}(y),z),
$
where $\varphi_z : \R^k \to \R^k$ is an invertible measure preserving transformation, with $\varphi_z(0)=0$, and
\begin{equation*} 
\omega_k |\varphi_z(y)|^{1/k} = \sigma_z(y).
\end{equation*}
Recall that $\sigma_z$ is given in Theorem~\ref{thm:inf1} (see also Remark~\ref{rem:mp}). In this case,
\begin{equation} \label{eq:decomp}
\Delta \tilde{u}_{*,k}(0)= \Delta_y \tilde u(\varphi_z^{-1}(y),z) + \Delta_z  \tilde u(\varphi_z^{-1}(y),z) \big|_{(y,z)=(0,0)}.
\end{equation}
For the first term, we use that
$$
\MA_k u(x) = \inf_{\psi \in \Psi} \Delta(\tilde u \circ \psi)(0)  , 
$$
where $\Psi = \{ \psi: \R^k \to \R^k~ \hbox{measure preserving such that}~\psi(0)=0 \big\}$,
and the fact that the infimum is attained when $\tilde{u}\circ \psi$ is a radially symmetric increasing function \cite{CS}. Hence,\begin{equation} \label{eq:first}
\Delta_y \tilde u(\varphi_z^{-1}(y),z) \big |_{(y,z)=(0,0)} = \MA_k u(x).
\end{equation}
For the second term, call $\phi(y,z)= (\varphi_z^{-1}(y),z)$ and compute:
\begin{align*}
 \Delta_z (\tilde{u}\circ \phi)(0) &= \tr\big(D_z\phi(0)^T D_z^2 \tilde{u}( \phi(0)) D_z\phi(0)  \big)
 + \nabla_z \tilde{u}(\phi(0))^T \cdot \Delta_z \phi(0).
\end{align*}
Recall that  $\phi(0)=0$ and
$
\tilde{u}(y,z) = u(x+(y,z)) -u(x) - \nabla_y u(x) \cdot y- \nabla_z u(x) \cdot z.$
 Then 
$$
\nabla_z \tilde{u}(\phi(0))=0,  \quad D_z^2 \tilde{u}(\phi(0)) = D_z^2{u}(x), \quad \hbox{and} \quad D_z\phi(0) = (0,I_{n-k}),
$$
where $I_{n-k}$ denotes the identity matrix in $M_{n-k}.$ Therefore,
\begin{equation}\label{eq:second}
\Delta_z  \tilde u(\varphi_z^{-1}(y),z) \big|_{(y,z)=(0,0)} = \Delta_z (\tilde{u}\circ \phi)(0) = \tr \big( D_z^2 u(x) \big) = \Delta_{n-k} u(x).
\end{equation}
Combining \eqref{eq:decomp}, \eqref{eq:first} and \eqref{eq:second}, we conclude \eqref{eq:limit}.

%%%%%%%%%%%%%%%%%%%%%%%%%%%%%%%%
\subsection{Connection to optimal transport} \label{sec:OT}
%%%%%%%%%%%%%%%%%%%%%%%%%%%%%%%%

In Corollary~\ref{cor:inf}, we obtained a representation of the function $\F_k^s u$ in terms of the $k$-symmetric increasing rearrangement. Using this representation, we find an equivalent expression of $\F_k^s u$ that can be understood from the viewpoint of optimal transport.

\begin{thm} \label{thm:repOT}
Suppose we are under the assumptions of Theorem~\ref{thm:inf1}. Then for any  $z\in \R^{n-k}$, $z\neq 0$, there exists an invertible map $\varphi_z : \R^k\to \R^k$ such that
\begin{equation}\label{eq:repOT}
\F_k^s u(x) = c_{n,s} \int_{\R^{n-k}} \int_{\R^k} \frac{\tilde{u}( \varphi_z^{-1}(y),z)}{\big(|y|^2+|z|^2\big)^\frac{n+2s}{2}}\, dydz.
\end{equation}
Moreover, if $\sigma_z: \R^k \to [0,\infty)$ is the Ryff's map given in Theorem~\ref{thm:inf1}, then $\varphi_z$ is measure preserving if and only if
\begin{equation} \label{eq:relation}
\omega_k |\varphi_z(y)|^k  =\sigma_z(y), \qquad \hbox{for}~ a.e.~y\in \R^k.
\end{equation}
\end{thm}

The key tool to prove Theorem~\ref{thm:repOT} is Brenier--McCann's theorem, a very well-known result in the theory of optimal transport  \cite{Brenier,MC}. We state it here in the form that we will use it.

\begin{thm} \label{thm:MC}
Let $f, g \in L^1(\R^k)$. Assume that 
$$\|f\|_{L^1(\R^k)}=\|g\|_{L^1(\R^k)}.$$ 
Then there exists a convex function $\psi: \R^k \to \R$ whose gradient $\nabla \psi$ pushes forward $f\, dy$ to $g \, dy$. Namely, for any measurable function $h$ in $\R^k$,
\begin{equation}\label{eq:MC3}
\int_{\R^k} h(y)g(y)\, dy = \int_{\R^k} h(\nabla \psi(y))f(y) \, dy.
\end{equation}
Moreover, $\nabla \psi: \R^k \to \R^k$ is invertible and unique.
\end{thm}

In the literature, $\nabla\psi$ is known as the (optimal) transport map.

\begin{proof}[Proof of Theorem~\ref{thm:repOT}]
Fix $z\in \R^{n-k}$, $z\neq 0$, and consider $f_z,g_z\in L^1(\R^k)$ given by
\begin{align*}
f_z(y) =  \big(|y|^{2}+|z|^2\big)^{-\frac{n+2s}{2}}\quad \hbox{and} \quad
g_z(y)& =  \big((\omega_k^{-1}\sigma_z(y))^{2/k}+|z|^2\big)^{-\frac{n+2s}{2}} ,
\end{align*}
where $\sigma_z: \R^k \to [0,\infty)$ is given in Theorem~\ref{thm:inf1}. 
Note that
\begin{align*}
\|f\|_{L^1(\R^k)} &= \int_{\R^k} \big((\omega_k^{-1}\sigma_z(y))^{2/k}+|z|^2\big)^{-\frac{n+2s}{2}} \, dy\\
&= k \omega_k \int_0^\infty \big( r^2 + |z|^2\big)^{-\frac{n+2s}{2}} r^{k-1} \, dr\\
&= \int_{\R^k} \big(|y|^{2}+|z|^2\big)^{-\frac{n+2s}{2}} \, dy = \|g\|_{L^1(\R^k)},
\end{align*}
since $\sigma_z$ is measure preserving.
By Theorem~\ref{thm:MC},
 there exists a convex function $\psi_z: \R^k \to \R$ (depending on $z$) whose gradient $\nabla \psi_z$ pushes forward $f_z\, dy$ to $g_z \, dy$.  Moreover, $\nabla\psi_z$ is invertible and unique. Call $\varphi_z = (\nabla \psi_z)^{-1}$.
Using \eqref{eq:MC3}, with $h(y)=\tilde u(y,z)$, we see that
 \begin{equation} \label{eq:MC2}
\int_{\R^k} \frac{\tilde u(y,z)}{\big((\omega_k^{-1}\sigma_z(y))^{2/k}+|z|^2\big)^{\frac{n+2s}{2}}}\, dy
=\int_{\R^k} \frac{\tilde u(\varphi^{-1}_z(y),z)}{\big(|y|^{2}+|z|^2\big)^{\frac{n+2s}{2}}}\, dy.
\end{equation}
Integrating over $z\in \R^{n-k}$, we obtain \eqref{eq:repOT}.

It remains to show that $\varphi_z$ is measure preserving if and only if \eqref{eq:relation} holds.
Indeed, for any measurable set $E \subset \R^k$, we have
\begin{align*}
\H^k \big( \varphi_z^{-1}(E)\big) 
&= \int_{\varphi_z^{-1}(E)}  dy
= \int_{\varphi_z^{-1}(E)}\frac{\big(|y|^2+|z|^2\big)^\frac{n+2s}{2}}{\big(|y|^2+|z|^2\big)^\frac{n+2s}{2}}\, dy\\
&= \int_{\varphi_z^{-1}(E)} \frac{\big(|\varphi_z(\varphi^{-1}_z (y))|^2+|z|^2\big)^\frac{n+2s}{2}}{\big(|y|^2+|z|^2\big)^\frac{n+2s}{2}}\, dy\\
&= \int_{E } \frac{\big(|\varphi_z(y)|^2+|z|^2\big)^\frac{n+2s}{2}}{\big(\omega_k^{-1}\sigma_z(y)\big)^{2/k}+|z|^2\big)^\frac{n+2s}{2}}\, dy,
\end{align*}
where the last equality follows from \eqref{eq:MC2} with $h(y)=\big(|\varphi_z(y)|^2+|z|^2\big)^\frac{n+2s}{2} \chi_E(y)$. Therefore, 
$$
\H^k \big( \varphi_z^{-1}(E)\big) = \H^{k}(E)
$$
if and only if
$\omega_k |\varphi_z(y)|^k  =\sigma_z(y),$  for $a.e.~y\in \R^k.$

\end{proof}

%%%%%%%%%%%%%%%%%%%%%%%%%%%%%%%%%%%%%
\section{Regularity of $\F_k^s u$} \label{sec:reg}
%%%%%%%%%%%%%%%%%%%%%%%%%%%%%%%%%%%%%

Given $x_0\in \R^n$,  we define the sections 
$$
D_{x_0} u(t) = \big\{ x \in \R^n : u(x)-u(x_0) - (x-x_0) \cdot \nabla u (x_0) \leq t\big\},  \quad \hbox{for}~ t>0.
$$
 Our main regularity result is the following.

\begin{thm} \label{thm:regularity}
Let $s\in (1/2,1)$ and $1\leq k <n$. Let $u\in C^{1,1}(\R^n)$ be convex.
Fix $x_0 \in \R^n$ and $r_0, \vep>0$.  Suppose that $\Lambda=\sup_{x\in B_{r_0}(x_0)} \diam(D_{x} u(\vep)) <\infty$ and
 $M=\sup_{x\in B_{r_0}(x_0)} \F_k^s u(x)<\infty$.
 Then $\F_k^s u \in C^{0,1-s}(\overline{B_{r}(x_0)})$ with $r< \min\{r_0/4, \Lambda, \vep/ (8\Lambda) \}$, and
$$
[\F_k^s]_{C^{0,1-s}(\overline{B_{r}(x_0)})} \leq C_0  [u]_{C^{1,1}(\R^n)}
$$
for some constant $C_0>0$ depending only on $n$, $k$,  $s$, $\vep$, $\Lambda$, and $M$.
\end{thm}

This theorem will be a consequence of the next proposition.

\begin{prop} \label{prop:lem4.6}
Fix $x_0 \in \R^n$ and $\vep>0$. 
Suppose that $\Lambda= \diam(D_{x_0} u(\vep)) <\infty$ and $[u]_{C^{1,1}(\R^n)}\leq 1$. Then for any $x_1 \in {B_{r}(x_0)}$, with $r\leq \vep/(4\Lambda)$, it holds that
$$
\F_k^s u(x_1) - \F_k^s u(x_0) \leq C \Lambda^{1-s} |x_1-x_0|^{1-s} +  \tfrac{4\Lambda}{\vep} |x_1-x_0| \F_k^s u(x_0)
$$
for some $C>0$ depending only on $n$, $k$, and $s$.
\end{prop}

First, we prove Theorem~\ref{thm:regularity}.

\begin{proof}[Proof of Theorem~\ref{thm:regularity}]
Without loss of generality, we may assume that $[u]_{C^{1,1}(\R^n)} \leq 1$. Otherwise, we consider $u/[u]_{C^{1,1}(\R^n)}$. 
Let $r< \min\{r_0/4, \Lambda, \vep/ (8\Lambda) \}$.
It is enough to show that
\begin{equation}\label{eq:estimate}
[ \F_k^s]_{C^{0,1-s}(\overline{B_{r}(x_0)})}  \leq C_0,
\end{equation}
for some constant $C_0>0$ depending only on $n$, $k$,  $s$, $\vep$, $\Lambda$, and $M$. 

Let $x_1,x_2 \in \overline{B_{r}(x_0)}$. Then $x_2 \in \overline{B_{2r}(x_1)} \subset B_{r_0}(x_0)$, since $4r<r_0$. Moreover, $\diam(D_{x_1}u(\vep))\leq \Lambda <\infty$. Hence, applying Proposition~\ref{prop:lem4.6} to $u$, replacing $B_{r}(x_0)$ by $B_{2r}(x_1)$, we get
\begin{align*}
\F_k^s u(x_2) - \F_k^s u(x_1) 
& \leq C \Lambda^{1-s} |x_2-x_1|^{1-s} +  \tfrac{4\Lambda}{\vep} |x_2-x_1| \F_k^s u(x_1) \leq C_0  |x_2-x_1|^{1-s},
\end{align*}
where $C_0=C \Lambda^{1-s} +  {4\Lambda^{1+s} M}/({\vep 2^s})$. Since $x_1$ and $x_2$ are arbitrary, we conclude \eqref{eq:estimate}.
\end{proof}

Before we prove Proposition~\ref{prop:lem4.6}, we need several preliminary results.

\begin{lem}\label{lem:layercake}
If $f$ is monotone increasing, then
$$
\int_0^\infty f(r) \omega(r) \, dr = \int_0^\infty \int_{\mu_f(t)}^\infty \omega (r)\, dr dt,
$$
with $\mu_f(t) = |\{r>0 : f(r) \leq t\}|$. 
\end{lem}

\begin{proof}
By Fubini's theorem, we have
$$
\int_0^\infty \int_{\mu_f(t)}^\infty \omega (r)\, dr dt = 
\int_0^\infty \omega(r)  \int_{\{r>\mu_f(t)\}} \, dt dr.
$$
Since $f$ is monotone increasing, then $r>\mu_f(t)$ if and only if $t< f(r)$. Therefore,
$$
\int_{\{r>\mu_f(t)\}} \, dt = \int_0^{f(r)} \, dt = f(r).
$$
\end{proof}

\begin{prop} \label{prop:repW}
Let $x\in \R^n$. Under the assumptions of Corollary~\ref{cor:inf} it holds that
$$
\F_k^s u(x) = c_{n,s} \int_0^\infty \int_{\R^{n-k}} \frac{1}{|z|^{n-k+2s}} W \Big( \frac{\mu_{x} u(t,z)^{1/k}}{|z|}\Big) \, dz dt ,
$$
where $\mu_{x} u (t,z)=\omega_k^{-1} \H^k \big( \{ y \in \R^k :  \tilde{u}_x (y,z) \leq t \} \big)$ and 
\begin{equation} \label{eq:W}
W(\rho) = k \omega_k \int_\rho^\infty \frac{r^{k-1}}{(1+r^2)^{\frac{n+2s}{2}}}\, dr.
\end{equation}
\end{prop}

\begin{proof}
By Corollary~\ref{cor:inf}, we have that
\begin{align*}
\F_k^s u(x) &= \Delta^s \tilde u_{*,k}(0)
= c_{n,s} \int_{\R^{n-k}} \frac{1}{|z|^{n+2s}} \Big ( \int_{\R^k} \frac{\tilde u_{*,k}(y,z)}{\big( ||z|^{-1} y|^2+1\big)^{\frac{n+2s}{2}}} \, dy\Big) \, dz \\
&=  c_{n,s} \int_{\R^{n-k}} \frac{1}{|z|^{n-k+2s}} \Big(k \omega_k \int_0^\infty  v(|z| r,z) \frac{r^{k-1}}{( r^2+1)^{\frac{n+2s}{2}}} \, dr\Big) \, dz, 
\end{align*}
where $v(r,z)= \tilde u_{*,k} (y,z)$ for $|y|=r$. 

Next we apply Lemma~\ref{lem:layercake} to $f(r)=v(|z| r,z)$ and $\omega(r)=  k \omega_k r^{k-1}(r^2+1)^{-\frac{n+2s}{2}}$. Note that since $v$ is the $k$-symmetric increasing rearrangement of $\tilde u$, we have
\begin{align*}
\mu_f(t)
&=  \tfrac{1}{|z|} |\{ r>0: v(r,z)<t\}| = \tfrac{\omega_k^{-1/k}}{|z|} \H^k \big( \{ y \in \R^k : \tilde u (y,z) < t\}\big)^{1/k}= \tfrac{1}{|z|}\mu_{x} u(t,z)^{1/k}.
\end{align*}
Therefore,
\begin{align*}
 k \omega_k \int_0^\infty  v(|z| r,z) \frac{r^{k-1}}{( r^2+1)^{\frac{n+2s}{2}}} \, dr
 &=  \int_0^\infty  \Big(k \omega_k \int_{\mu_{x} u(t,z)^{1/k}/|z|}^\infty \frac{r^{k-1}}{( r^2+1)^{\frac{n+2s}{2}}} \, dr\Big)dt\\
& = \int_0^\infty W \Big( \frac{\mu_{x} u(t,z)^{1/k}}{|z|}\Big) \, dt, 
\end{align*}
where $W$ is given in \eqref{eq:W}.
By Fubini's theorem, we conclude that
$$
\F_k^s u(x) = c_{n,s} \int_0^\infty \int_{\R^{n-k}} \frac{1}{|z|^{n-k+2s}} W \Big( \frac{\mu_{x} u(t,z)^{1/k}}{|z|}\Big) \, dz dt. 
$$
\end{proof}

\begin{lem} \label{lem:lem4.5}
Suppose we are under the assumptions of Proposition~\ref{prop:lem4.6}. Let $x_1 \in \overline{B_r(x_0)}$ and  $d=|x_1-x_0|$.
The following holds:
\begin{enumerate}[$(a)$]
\item If $t \in (2\Lambda d, \vep]$, then
$
D_{x_0}u (t- 2\Lambda d) \subset D_{x_1} u (t). 
$
\item If $ t\in (\vep, \infty)$, then 
$
D_{x_0}u \big (t - 2\Lambda d{t}/{\vep} \big ) \subset D_{x_1} u (t). 
$
\end{enumerate}
\end{lem}

\begin{proof}
First we prove $(a)$. Fix $t \in  (2\Lambda d, \vep]$ and let $x\in D_{x_0}u(t-2\Lambda d)$. Then
\begin{equation} \label{eq1}
u(x)-u(x_0) - (x-x_0) \cdot \nabla u (x_0) \leq t-2\Lambda d.
\end{equation}
Using \eqref{eq1}, convexity, and $[u]_{C^{1,1}(\R^n)}\leq 1$, we see that
\begin{align*}
u(x)-u(x_1) - (x-x_1) \cdot \nabla u (x_1) 
&= u(x)-u(x_0) - (x-x_0) \cdot \nabla u (x_0) \\
& \quad - \big( u(x_1) - u(x_0) -(x_1-x_0) \cdot \nabla u(x_0) \big)\\
& \quad + (x-x_1) \cdot (\nabla u(x_0) - \nabla u(x_1))\\
&\leq t - 2\Lambda d + |x-x_1| d.
\end{align*}
Moreover, $x\in D_{x_0} u(\vep)$, since $t \leq \vep$, and thus,
$$
|x-x_1| \leq |x-x_0| + |x_0 - x_1| \leq \Lambda + d \leq 2\Lambda .
$$
Therefore, $x\in D_{x_1} u(t)$.

Next we prove $(b)$. Fix $t\in (\vep,\infty)$ and let $x\in D_{x_0}u \big (t - 2\Lambda dt/\vep \big )$. By the previous computation, we have that
\begin{align} \label{eq2}
u(x)-u(x_1) - (x-x_1) \cdot \nabla u (x_1) & \leq t - 2\Lambda dt/\vep + (|x-x_0|+\Lambda) d.
\end{align}
To control the distance from $x$ to $x_0$, we need to estimate the diameter of $D_{x_0} u(t)$. 
Take $y\in D_{x_0}u(t) \setminus D_{x_0}u(\vep)$ and let $z$ be in the intersection between $\partial D_{x_0}u(\vep)$ and the line segment joining $x_0$ and $y$. Then there is some $\lambda>1$ such that $y-x_0=\lambda (z-x_0)$. 
By convexity of $u$, 
$$
u(z) \leq \tfrac{\lambda-1}{\lambda} u(x_0) + \tfrac{1}{\lambda} u(y).
$$
Therefore,
\begin{align*}
\lambda \vep &= \lambda\big (u(z) - u(x_0) -(z-x_0) \cdot \nabla u(x_0) \big)\\
& \leq (\lambda-1) u(x_0) + u(y) -\lambda u(x_0) -(y-x_0)\cdot \nabla u(x_0)\\
& = u(y) - u(x_0) - (y-x_0) \cdot \nabla u(x_0) \leq t,
\end{align*}
so $\lambda \leq {t}/{\vep}$. By convexity, we have that $D_{x_0}u(t) \subset  x_0 + \tfrac{t}{\vep} (D_{x_0}u(\vep)-x_0)$.
It follows that
$$
\diam D_{x_0} u(t) \leq {t}/{\vep} \diam D_{x_0} u(\vep) = \Lambda t / \vep.
$$
Hence, $|x-x_0| \leq \Lambda t/\vep$, and by \eqref{eq2}, we get
$$
u(x)-u(x_1) - (x-x_1) \cdot \nabla u (x_1) \leq t - 2\Lambda d{t}/{\vep} + (\Lambda {t}/{\vep}+\Lambda) d \leq t,
$$
which means that $x \in D_{x_1} u (t). $
\end{proof}

We are ready to give the proof of Proposition~\ref{prop:lem4.6}.

\begin{proof}[Proof of Proposition~\ref{prop:lem4.6}]
Let $x_1 \in {B_{r}(x_0)}$, with $r\leq \vep/(4\Lambda)$, and call $d=|x_0-x_1|$. We will estimate $\F_k^s u(x_1)$ using Proposition~\ref{prop:repW}:
$$
\F_k^s u(x_1) 
= c_{n,s} \int_0^\infty \int_{\R^{n-k}} \frac{1}{|z|^{n-k+2s}} W \Big( \frac{\mu_{x_1} u(t,z)^{1/k}}{|z|}\Big) \, dz dt. 
$$
In view of Lemma~\ref{lem:lem4.5}, we divide the integral with respect to $t$ in three parts:
 $$
 {\rm I.}\  t \in (0,2\Lambda d], \quad 
 {\rm  II.} \ t\in (2\Lambda d, \vep],\quad
 {\rm  III.}\  t\in (\vep, \infty).
$$
Let us start with I. Since $u\in C^{1,1}(\R^n)$ with $[u]_{C^{1,1}(\R^n)} \leq 1$, then
$$
\mu_{x_1} u(t,z) \geq (t-|z|^2)_+^{k/2}.
$$
Hence, using that $W(\rho)$ is monotone decreasing, we get
$$
W \Big( \frac{\mu_{x_1} u(t,z)^{1/k}}{|z|}\Big) \leq W \Big( \big(\tfrac{t}{|z|^2}-1\big)_+^\frac{1}{2}\Big).
$$
Therefore,
\begin{align*}
\int_{\R^{n-k}} \frac{1}{|z|^{n-k+2s}} W \Big( \frac{\mu_{x_1} u(t,z)^{1/k}}{|z|}\Big) \, dz
&\leq  \int_{\{|z|< t^{1/2}\}} \frac{1}{|z|^{n-k+2s}} W \Big( \big(\tfrac{t}{|z|^2}-1\big)^\frac{1}{2}\Big) \, dz\\
&\qquad + W(0) \int_{\{|z|> t^{1/2}\}} \frac{1}{|z|^{n-k+2s}}  \, dz \equiv I_1+I_2.
\end{align*}
Note that $W(0)=C(n,k,s)<\infty$. Then
$$
I_2 \lesssim \int_{t^{1/2}}^{\infty} \frac{1}{\rho^{n-k+2s}} \rho^{n-k-1} \, d\rho \eqsim t^{-s}.
$$
For $I_1$, we make the change of variables, $w=z/t^{1/2}$. We see that
\begin{align*}
 I_1 &= \int_{\{|w|< 1\}} \frac{1}{t^\frac{n-k+2s}{2} |w|^{n-k+2s}} W \big( \big(\tfrac{1}{|w|^2}-1\big)^\frac{1}{2}\big) t^\frac{n-k}{2} \, dw
 \eqsim \frac{1}{t^s} \int_0^1 \frac{1}{\rho^{1+2s}} W\big(\big( \tfrac{1}{\rho^2}-1\big)^\frac{1}{2}\big) \, d\rho.
\end{align*}
Note that if $0<\rho \leq 1/2$, then $\big( \tfrac{1}{\rho^2}-1\big)^\frac{1}{2}\geq \tfrac{1}{\sqrt 2\rho}$. Hence,
\begin{align*}
W\big(\big( \tfrac{1}{\rho^2}-1\big)^\frac{1}{2}\big) \leq W\big( \tfrac{1}{\sqrt{2}\rho}\big) = \int_{\tfrac{1}{\sqrt{2}\rho}}^\infty \frac{r^{k-1}}{(1+r^2)^\frac{n+2s}{2}}\, dr \lesssim \rho^{n-k+2s}.
\end{align*}
Therefore,
$$
I_1 \lesssim  t^{-s} \int_0^{1/2} \frac{1}{\rho^{1+2s}} \rho^{n-k+2s} \, d\rho + t^{-s} W(0) \int_{1/2}^1  \frac{1}{\rho^{1+2s}} \, d\rho \eqsim t^{-s},
$$
since $n-k > 0$. We conclude that
\begin{align*}
{\rm I} &= c_{n,s} \int_0^{2\Lambda d} \int_{\R^{n-k}} \frac{1}{|z|^{n-k+2s}}W \Big( \frac{\mu_{x_1} u(t,z)^{1/k}}{|z|}\Big) \, dz dt\\ 
&\lesssim \int_0^{2\Lambda d} t^{-s}\, dt 
\eqsim (2\Lambda d)^{1-s} = (2\Lambda)^{1-s} |x_1-x_0|^{1-s}.
\end{align*}

Next we estimate the integral for $t\in(2\Lambda d, \vep]$. To this end, we use Lemma~\ref{lem:lem4.5}, part~$(a)$:
$$
D_{x_0}u (t- 2\Lambda d) \subset D_{x_1} u (t). 
$$
In particular, for any $z\in \R^{n-k}$ fixed, we have 
$$
\{ y \in \R^k : \tilde u_{x_0}(y,z) \leq t-2\Lambda d \} \subset \{ y \in \R^k : \tilde u_{x_1}(y,z) \leq t \} .
$$
Hence, $\mu_{x_0}(t-2\Lambda d, z) \leq \mu_{x_1}(t,z)$, which yields
\begin{align*}
{\rm II} &= c_{n,s} \int_{2\Lambda d}^\vep  \int_{\R^{n-k}}\frac{1}{|z|^{n-k+2s}} W \Big( \frac{\mu_{x_1} u(t,z)^{1/k}}{|z|}\Big) \, dz dt \\
&\leq c_{n,s} \int_{0}^{\vep-2\Lambda d}  \int_{\R^{n-k}}\frac{1}{|z|^{n-k+2s}} W \Big( \frac{\mu_{x_0} u(t,z)^{1/k}}{|z|}\Big) \, dz dt.
\end{align*}
Finally, we estimate the integral for $t\in[\vep,\infty)$. By  Lemma~\ref{lem:lem4.5}, part~$(b)$:
$$
D_{x_0}u \big (t - 2\Lambda d{t}/{\vep} \big ) \subset D_{x_1} u (t). 
$$
Hence, $\mu_{x_0}u(t-2\Lambda d{t}/{\vep}, z) \leq \mu_{x_1}u(t,z)$, and
\begin{align*}
{\rm III} &= c_{n,s} \int_\vep^{\infty} \int_{\R^{n-k}} \frac{1}{|z|^{n-k+2s}}W \Big( \frac{\mu_{x_1}u(t,z)^{1/k}}{|z|}\Big) \, dz dt\\ 
& \lesssim  \int_\vep^{\infty} \int_{\R^{n-k}} \frac{1}{|z|^{n-k+2s}}W \Big( \frac{\mu_{x_0}u(t-2\Lambda d{t}/{\vep}, z)^{1/k}}{|z|}\Big) \, dz dt\\
& = \frac{1}{1-2\Lambda d/\vep} \int_{\vep-2\Lambda d}^{\infty}  \int_{\R^{n-k}}\frac{1}{|z|^{n-k+2s}} W \Big( \frac{\mu_{x_0} u(t,z)^{1/k}}{|z|}\Big) \, dz dt.
\end{align*}
Note that
$$
{\rm II+III} \leq  \frac{c_{n,s}}{1-2\Lambda d/\vep}\int_{0}^{\infty}  \int_{\R^{n-k}}\frac{1}{|z|^{n-k+2s}} W \Big( \frac{\mu_{x_0} u(t,z)^{1/k}}{|z|}\Big) \, dz dt = \frac{\vep}{\vep-2\Lambda d} \F_k^s u(x_0).
$$
Therefore, we conclude that
\begin{align*}
\F_k^s u(x_1)- \F_k^s u(x_0) 
&\leq C \Lambda^{1-s} |x_1-x_0|^{1-s} 
+ \big(\tfrac{\vep}{\vep-2\Lambda d}-1\big) \F_k^s u(x_0)\\
& \leq  C \Lambda^{1-s} |x_1-x_0|^{1-s} 
+ \tfrac{4\Lambda}{\vep} |x_1-x_0| \F_k^s u(x_0)
\end{align*}
since $d < r \leq \vep/(4\Lambda)$, and thus, $\vep - 2\Lambda d \geq \vep /2$.
\end{proof}

 %%%%%%%%%%%%%%%%%%%%%%%%%%%%%%%%%%%%%
\section{A Global Poisson Problem} \label{sec:GPP}
%%%%%%%%%%%%%%%%%%%%%%%%%%%%%%%%%%%%%

We consider the following Poisson problem in the full space:
 \begin{align} \label{eq:PP}
 \begin{cases}
 \F_k^s u = u - \varphi & \hbox{in}~\Rn\\
 (u-\varphi)(x) \to 0 & \hbox{as}~|x|\to \infty,
 \end{cases}
\end{align}
where $\varphi:\R^n \to \R$ is nonnegative, smooth, and strictly convex. Furthermore, we ask that $\varphi$ behaves asymptotically at infinity as a cone $\phi$, that is,  
\begin{equation} \label{eq:cone}
\lim_{|x|\to\infty} (\varphi-\phi)(x) =0.
\end{equation}
Similar problems have been studied for nonlocal Monge-Amp\`{e}re operators in \cite{CC,CS}.

We will prove the following theorem.
\begin{thm} \label{thm:mainPP}
There exists a unique solution $u$ to \eqref{eq:PP} such that $u \in C^{1,1}(\R^n)$ with
$$
[u]_{C^{1,1}(\R^n)} \leq [\varphi]_{C^{1,1}(\R^n)}.
$$
\end{thm}

To define the notion of solution, we introduce a natural pointwise definition of $\F_k^s u$ for functions $u$ that are merely continuous. 

\begin{defn}
Let $u \in C^0(\Rn)$.
\begin{enumerate}[$(a)$]
\item We say that a linear function $l(y)= y \cdot p + b$, with $p\in \Rn$, and $b\in \R$, is a supporting plane of $u$ at a point $x$ if $l(x)=u(x)$ and $l(y)\leq u(y)$, for all $y\in \Rn$. 
\item We define the subdifferential of $u$ at a point $x$ as the set $\partial u (x)$ of all vectors $p\in \Rn$ such that $l(y)= y \cdot p +b$ is a supporting plane of $u$ at $x$, for some $b\in \R$.
\end{enumerate}
\end{defn}

\begin{defn}
Let $u \in C^0(\Rn)$ be a convex function. For $x_0\in \Rn$, we define
\begin{equation*}
\F_k^s u(x_0) = c_{n,s} \sup_{p \in \partial u(x_0)} \inf_{K\in \K_{k}^s} \int_{\Rn}  (u(x_0+x)-u(x_0)-x\cdot p) K(x)\, dx.
\end{equation*}
\end{defn}

\begin{rem}
Note that if $u\in C^{1,1}(x_0)$, then $\partial u(x_0) = \{\nabla u(x_0)\}$, and the previous definition coincides with Definition~\ref{def:IO}.
\end{rem}

The following properties of $\F_k^s u$ will be useful for our purposes. The proof is analogous to the one in \cite{CS}, so we omit it here.

\begin{lem} \label{lem:properties}
Let $u,v \in C^0(\Rn)$ be convex functions.The following holds:
\begin{enumerate}[$(a)$]
\item (Homogeneity). For any $\lambda >0$, 
$$
\F_k^s (\lambda u) = \lambda \F_k^s u.
$$ 

\item (Monotonicity).  Assume that $u(x_0)=v(x_0)$ and $u(x)\geq v(x)$ for all $x\in \Rn$. Then
$$
\F_k^s u(x_0) \geq \F_k^s v(x_0).
$$
\item (Concavity). For any $x\in \Rn$,
$$
\F_k^s \Big ( \frac{u+v}{2}\Big) (x) \geq \frac{\F_k^s u(x) + \F_k^s v (x)}{2}.
$$
 \item (Lower semicontinuity). Assume that $u\in C^{1,1}(\Rn)$. Then
$$
\F_k^s u(x_0) \leq \liminf_{x\to x_0} \F_k^s u(x).
$$
\end{enumerate}
\end{lem}

\begin{defn}
Let $u \in C^0(\Rn)$ be a convex function. We say that $u$ is a subsolution to $ \F_k^s u = u - \varphi$ in $\Rn$ if 
$$
\F_k^s u (x_0) \geq u(x_0) - \varphi(x_0), \quad \hbox{for all}~x_0 \in \Rn.
$$
Similarly, $u$ is a supersolution if 
$$
\F_k^s u (x_0) \leq u(x_0) - \varphi(x_0),  \quad \hbox{for all}~x_0 \in \Rn.
$$
We say that $u$ is a solution if it is both a subsolution and a supersolution.
\end{defn}

\begin{lem} \label{lem:maxsub}
If $u$ and $v$ are subsolutions, then $\max\{u,v\}$ is a subsolution. 
\end{lem}

\begin{proof}
Let  $w=\max\{u,v\}$. Then $w$ is continuous and convex. Fix $x_0 \in \Rn$.  Without loss of generality, we may assume that $u(x_0) \geq v(x_0)$. Then $w(x_0)=u(x_0)$ and $w(x) \geq u(x)$, for any $x\in \Rn$. By monotonicity (see Lemma~\ref{lem:properties}), we have
$$
\F_k^s w(x_0) \geq \F_k^s u (x_0) \geq u(x_0) - \varphi(x_0) = w(x_0) - \varphi(x_0).
$$
Hence, $w$ is a subsolution.
\end{proof}

We will show existence and uniqueness of solutions to \eqref{eq:PP} using Perron's method. The key ingredients are the comparison principle, and the existence of a subsolution (lower barrier) and a supersolution (upper barrier). We state this in the following proposition. We omit the proof since it is similar to that in \cite{CS}.

\begin{prop} \label{prop:tools}
Consider the equation $\F_k^s u = u - \varphi$ in $\Rn$.
The following holds:
\begin{enumerate}[$(a)$]
\item (Comparison principle).
Let $u$ and $v$ be a subsolution and supersolution, respectively.  Assume that $u\leq v$ in $\R^n\setminus \Omega$, for some bounded domain $\Omega \subset \R^n$. Then $u\leq v$ in $\Rn$.\medskip

\item (Lower-barrier). The function $\varphi$ is a subsolution.\medskip

\item (Upper-barrier). The function $\varphi+w$ is a supersolution, where $w= (I-\Delta^s)^{-1}\Delta^s \varphi$. In particular, $w(x) \leq C(1+|x|)^{1-2s}$, for some $C>0$.
\end{enumerate}
\end{prop}

An immediate consequence of the comparison principle is the uniqueness of solutions.

 \begin{lem}[Uniqueness]
There exists at most one solution to \eqref{eq:PP}. 
 \end{lem}
 
 \begin{proof}
 Suppose by means of contradiction that there exist two functions $u,v \in C^0(\R^n)$ with $u\neq v$, satisfying \eqref{eq:PP}. Then
 $|u(x) -v(x)|\to 0$, as $|x|\to \infty$. Hence, for any $\vep>0$, there exists a compact set $\Omega_\vep \in \R^n$, depending on $\vep$, such that
 $$
 v(x)-\vep \leq u(x) \leq v(x)+ \vep \quad \hbox{for all}~x\in \R^n\setminus \Omega_\vep.
 $$
 Moreover,   for any $x_0 \in \R^n$, the function $v+\vep$ satisfies
 $$
 \F_k^s (v+\vep)(x_0) = v(x_0)- \varphi(x_0) <(v(x_0)+\vep) - \varphi(x_0).
 $$
Therefore, $v$ is a supersolution and by the comparison principle, it follows that 
 $u \leq v+\vep$  in $\Rn.$
 Similarly, we see that $v-\vep$ is a subsolution and $u \geq v - \vep$ in $\Rn.$ Hence,
 $$\|u-v\|_{L^\infty(\Rn)}\leq \vep,$$ 
 and letting $\vep\to0$, we get $u=v$ in $\Rn$, which is a contradiction.
 \end{proof}

To prove existence of a solution, we define
\begin{equation} \label{eq:perron}
u(x) = \sup_{v \in \S} v(x),
\end{equation}
where $\S$ is the set of admissible subsolutions given by
$$
\S =\big \{ v \in C^{0,1}(\Rn) : v~\hbox{subsolution}, \  \varphi \leq v \leq \varphi + w, \ \hbox{and} \ [v]_{C^{0,1}(\Rn)}\leq [\varphi]_{C^{0,1}(\R^n)} \big \}.
$$
Note that $\S \neq \emptyset$ since $\varphi \in \S$, and the supremum is finite since $v\leq \varphi+w$, for any $v\in \S$. Moreover, $u$ is convex and Lipschitz, with 
$$
[u]_{C^{0,1}(\Rn)} \leq [\varphi]_{C^{0,1}(\Rn)}.
$$
From $\varphi \leq u \leq \varphi+w$, and the upper bound for $w$ in Proposition~\ref{prop:tools}, it follows that
$$
0\leq (u- \varphi)(x) \leq w(x) \leq C(1+|x|)^{1-2s} \to 0,
$$
as $|x|\to \infty $ since $1-2s<0$.

\begin{prop} \label{prop:regu}
The function $u$ given in \eqref{eq:perron} is $C^{1,1}(\Rn)$ with 
$$
[u]_{C^{1,1}(\R^n)} \leq [\varphi]_{C^{1,1}(\R^n)}.
$$
\end{prop}

\begin{proof}
We will show that for any $x_0,x_1\in \Rn$,
$$
0\leq u(x_0+x_1) - u(x_0-x_1) - 2 u(x_0) \leq [\varphi]_{C^{1,1}(\R^n)} |x_1|^2.
$$
 Indeed, the lower bound follows from convexity of $u$. Hence, we only need to prove the upper bound. Call $M=[\varphi]_{C^{1,1}(\R^n)}$. Then
\begin{equation}\label{eq:barrier}
 \varphi(x_0+x_1) - \varphi(x_0-x_1) -M|x_1|^2 \leq 2 \varphi(x_0).
\end{equation}
 Take any $v \in \S$ and fix $x_1\in \Rn$. Define
 $$
 \hat{v}(x_0)= \tfrac{1}{2} \big( v(x_0+x_1)+v(x_0-x_1) - M|x_1|^2\big), \quad \hbox{for}~x_0\in \Rn.
 $$
We claim that $\hat v$ is a subsolution to  $\F_k^s u = u-\varphi$ in $\Rn$.
Indeed, since $\F_k^s$ is homogeneous of degree 1, concave, and translation invariant (see Lemma~\ref{lem:properties}), we have
\begin{align*}
\F_k^s \hat{v}(x_0) &= \F_k^s \Big (\tfrac{1}{2}v(x_0+x_1)+\tfrac{1}{2}v(x_0-x_1)\Big)\\
&\geq \tfrac{1}{2} \F_k^s v(x_0+x_1)+ \tfrac{1}{2} \F_k^s v(x_0-x_1)\\
&\geq \tfrac{1}{2} \big(v(x_0+x_1)- \varphi(x_0+x_1) + v(x_0-x_1)-\varphi(x_0-x_1) \big)\\
&= \tfrac{1}{2}\big(v(x_0+x_1)-v(x_0-x_1)- M |x_1|^2\big)- \tfrac{1}{2}\big( \varphi(x_0+x_1)+\varphi(x_0-x_1)-M|x_1|^2\big)\\
& \geq \hat v(x_0)- \varphi(x_0).
\end{align*}
Moreover, using that $v\leq\varphi+w$, we get
\begin{align*}
\hat{v}(x_0)
&\leq \tfrac{1}{2} \big(\varphi(x_0+x_1) +  \varphi(x_0-x_1) - M |x_1|^2\big) + \tfrac{1}{2}\big( w(x_0+x_1)+ w(x_0-x_1) \big).
\end{align*}
By \eqref{eq:barrier} and the upper bound of $w$ in Proposition~\ref{prop:tools}, part (c), we see that
$$
\hat{v}(x_0) - \varphi(x_0) \leq  \tfrac{C}{2}( 1+|x_0+x_1|^{1-2s}) + \tfrac{C}{2}(1+ |x_0-x_1|)^{1-2s}) \to 0,
$$
as $|x_0|\to \infty$ and $x_1$ fixed, since $1-2s<0$. Then for all $\vep>0$, there is some compact set $\Omega_\vep$, depending on $\vep$ and $x_1$, such that
$$
\hat{v}(x_0) - \vep \leq \varphi(x_0), \quad \hbox{for all}~x_0 \in \Rn \setminus \Omega_\vep.
$$
Consider $\hat v_\vep= \max\{ \hat v - \vep, \varphi\}$. Then $\hat v_\vep$ is a subsolution, since the maximum of subsolutions is a subsolution (see Lemma~\ref{lem:maxsub}). Also, $\hat{v}_\vep = \varphi \leq \varphi+w$ in $\R^n\setminus \Omega_\vep$, and $\varphi+w$ is a supersolution by Proposition~\ref{prop:tools}, part (c). Applying the comparison principle, we get $\varphi\leq \hat{v}_\vep \leq \varphi + w$. Moreover,  $[\hat v_\vep]_{C^{0,1}(\R^n)} \leq [\varphi]_{C^{0,1}(\R^n)}$. Therefore, $\hat v_\vep \in \S$. 

Since $u(x_0)= \sup_{v\in \S} v(x_0)$, it follows that 
$
u(x_0) \geq \hat v_\vep(x_0) \geq \hat v (x_0) - \vep.
$
Letting $\vep \to 0$, we conclude that for any $v\in \S$ and $x_0, x_1 \in \Rn$,
\begin{equation}\label{eq:anysub}
u(x_0) \geq \tfrac{1}{2}\big( v(x_0+x_1)+v(x_0-x_1) - M |x_1|^2 \big) .
\end{equation}
Finally, by definition of supremum, for any $\delta>0$, and $x_0,x_1\in \Rn$, there exist $v_1,v_2 \in \S$ such that $u(x_0+x_1) -\delta < v_1(x_0+x_1)$ and $u(x_0-x_1) - \delta < v_2 (x_0-x_1)$. Let $v=\max \{v_1, v_2\}$. Then using \eqref{eq:anysub} for this $v$, we get
$$
u(x_0) \geq \tfrac{1}{2}\big ( u(x_0+x_1)-\delta+u(x_0-x_1)-\delta - M |x_1|^2 \big).
$$
Letting $\delta \to 0$, we conclude that
$$
 u(x_0+x_1) - u(x_0-x_1) - 2 u(x_0) \leq [\varphi]_{C^{1,1}(\R^n)} |x_1|^2.
$$
\end{proof}

To complete the proof of Theorem~\ref{thm:mainPP}, it remains to see that $u$ is a solution. Hence, we need to show that $u$ is both a subsolution and a supersolution. We will prove these results in the next two propositions. 

\begin{lem} \label{lem:compact}
For any $x_0\in \Rn$ and  $\vep>0$, the set 
$$D_{x_0} u(\vep)=\big\{ x \in \R^n : u(x)-u(x_0) - (x-x_0) \cdot \nabla u (x_0) \leq \vep\big\}$$ 
is compact.
\end{lem}

\begin{proof}
Let $x_0\in \Rn$ and $\vep >0$. Without loss of generality, we may assume that $x_0=0$. Let $l$ be the supporting plane of $u$ at $0$, that is, 
$l(x)=u(0)+x \cdot \nabla u(0).$
Clearly, $D_{x_0} u(\vep)$ is closed. Hence, we only need to show that it is bounded. Recall that 
\begin{equation}\label{eq:strictsep}
\phi(x) < \varphi(x) \leq u(x), \quad \hbox{for all}~x\in \Rn,
\end{equation}
where $\phi$ is a cone. 
Note that the strict inequality in \eqref{eq:strictsep} follows from the strict convexity of $\varphi$. 
Moreover, by \eqref{eq:PP} and \eqref{eq:cone}, we have
$$
\lim_{|x|\to\infty} (u-\phi)(x) = 0.
$$
Therefore, $D_{x_0} u(\vep) \subset \{ \phi < l +\vep \}.$
We claim that
\begin{equation}\label{eq:limitcond}
\lim_{|x|\to\infty} (\phi-l)(x) =\infty.
\end{equation}
If this condition holds, then for all $M>0$, there exists $R>0$, such that 
$$
\phi(x) - l(x) > M, \quad \hbox{for all}~|x|>R.
$$
Choosing $M=\vep$, we see that $\{\phi < l+\vep\}\subset B_R$, for some $R$ depending on $\vep$. Hence, the set $D_{x_0} u(\vep)$  is  bounded.

To prove the claim, we distinguish two cases. If $u(0)=0$, then $u$ attains an absolute minimum at $0$, so $\nabla u(0)=0$. In particular, $l(x)=0$, for all $x\in \Rn$, and thus, \eqref{eq:limitcond} is clearly satisfied.
Hence, it remains to show the claim when 
$$
u(0)>0.
$$
We will prove it by contradiction. If \eqref{eq:limitcond} is not true, then there exists a sequence of points $\{x_j\}_{j=1}^\infty \subset \Rn$  such that $|x_j|\to\infty$, as $j\to\infty$, and
$$
\lim_{j\to\infty} ( \phi-l ) (x_j)  <\infty. 
$$
Using that $\phi$ is continuous and homogeneous of degree 1, and letting $j\to \infty$, we get
$$
\tfrac{\phi(x_j)}{|x_j|} - \tfrac{l(x_j)}{|x_j|} = \phi \big( \tfrac{x_j}{|x_j|}\big) - \tfrac{u(0)}{|x_j|} -   \tfrac{x_j}{|x_j|} \cdot \nabla u(0) \to \phi(e) - D_e u (0) =0,
$$
where $x_j/|x_j|\to e$, up to a subsequence. Therefore, $\phi(e)=D_e u(0)$. For any $\lambda>0$, we have
$$
l(\lambda e) = u(0) + \lambda e \cdot \nabla u(0)   =u(0) + \lambda   \phi( e) = u(0) +    \phi(\lambda e).
$$
Since $l$ is a supporting plane of $u$, we know that $u(x) \geq l(x)$, for all $x\in \Rn$, and thus,
$$
u(\lambda e) \geq l(\lambda e) = \phi(\lambda e) + u(0).
$$
Letting $\lambda\to \infty$,  we see that
$$
0 = \lim_{\lambda \to \infty} (u- \phi) (\lambda e)  \geq u(0) >0,
$$
which is a contradiction.

\end{proof}

\begin{prop}[$u$ is a subsolution] \label{prop:subsolution}
The function $u$ given in \eqref{eq:perron} satisfies
$$
\F_k^s u (x_0) \geq u(x_0) - \varphi(x_0), \quad \hbox{for all}~x_0 \in \Rn.
$$
\end{prop}

\begin{proof}
By Proposition~\ref{prop:regu}, we know that $u\in C^{1,1}(\Rn)$. 
Without loss of generality, we may assume that $[u]_{C^{1,1}(\Rn)}=1$. Otherwise, consider $u/[u]_{C^{1,1}(\Rn)}$.

Let $x_0 \in \Rn$. Then the quadratic polynomial
$$
P(x)=u(x_0)+\nabla u(x_0) \cdot (x-x_0) + |x-x_0|^2
$$
touches $u$ from above at $x_0$. Moreover, we may assume that $P$ touches $u$ strictly from above at $x_0$. If not, we replace $P$ by $P+\vep|x-x_0|^2$ with $\vep>0$ small. 

Fix $\delta >0$. Then there exists $h>0$, with $h\to 0$ as $\delta\to 0$, such that
$$
P(x) - u(x) \geq h >0, \quad \hbox{for all}~x\in \Rn \setminus B_\delta (x_0).
$$
Since $u(x)=\sup_{v\in \S} v(x)$ and $v\in \S$ is uniformly continuous, 
there is a monotone sequence $\{v_j\}_{j=1}^\infty \subset \S$ such that $v_j \to u$ uniformly in compact subsets of $\Rn$. In particular, there exists $j_0\geq 1$, depending on $h$, such that for all $j> j_0$,
\begin{equation} \label{eq:uniformconv}
u(x) - h < v_j(x), \quad \hbox{for all}~x\in \overline{B_\delta(x_0)}.
\end{equation}
Call $v=v_j$ for some $j>j_0$. It follows that  
$$
\begin{cases}
P- v\geq h &  \hbox{in}~\Rn\setminus B_\delta(x_0)\\
P-v < P-u + h & \hbox{in}~B_\delta(x_0).
\end{cases}
$$
Let $d= \inf_{\Rn} (P-v)$. Then $d=P(x_1)-v(x_1)$, for some $x_1 \in \overline{B_h(x_0)}$, with $0\leq d<h$, and
$$
\begin{cases}
P(x_1) - d = v(x_1) &\\
P(x) - d \geq v(x), & \hbox{for all}~x\in \Rn.
\end{cases}
$$
Hence, $P-d$ is a quadratic polynomial that touches $v$ from above at $x_1$. In particular, since $v$ is convex, then $v$ has a unique supporting plane $l$ at $x_1$, so $\partial v (x_1) = \{ \nabla l\}$.

Let $\tau \geq 0$ be such that $l+\tau$ is the supporting plane of $u$ at some point $x_2$. Note that $x_2$ approaches $x_0$ as $h$ goes to 0, and thus, there exists some $r=r(h)>0$ such that $r\to 0$, as $h\to0$, and $x_2 \in B_{r}(x_0)$. Furthermore, since
$
l(x_1)+d=v(x_1)+d = P(x_1) \geq u(x_1),
$
then $\tau \leq  d<h$ (see Figure~\ref{fig:parabolas}).
\begin{figure}[h]
\begin{tikzpicture}[domain=-4:4, scale = 2]
  \draw[black, line width = 0.05mm, dash pattern=on 1pt off 1pt]   plot[smooth,domain=0.5:1.5] (\x, {5*(\x-1.002)^2-0.515}); %P-d
    \draw[black, line width = 0.1mm]   plot[smooth,domain=-1:2.5] (\x, {0.3*(\x-0.5)^2-0.6}); %v
       \draw[black, line width = 0.1mm]   plot[smooth,domain=-1:2] (\x, {0.37*(\x-0.25)^2-0.25}); %u
     \draw[black, line width = 0.05mm]   plot[smooth,domain=0.5:1.5] (\x, {5*(\x-1.002)^2-0.017}); %P
       \draw[black, line width = 0.05mm]   plot[smooth,domain=-1:3] (\x,{0.306*(\x-1.01)-0.52}); %l
           \draw[black, line width = 0.05mm, dash pattern=on 1pt off 1pt]   plot[smooth,domain=-1:3] (\x,{0.306*(\x-1.01)-0.085}); %l+tau
      \draw [fill=black] (1.06,0) circle (0.5pt); %x_0
         \draw [fill=black] (1.015,-0.52) circle (0.5pt); %x_1
            \draw [fill=black] (0.7,-0.175) circle (0.5pt); %x_2
            
            \draw[color=black] (2,1) node {\small $u$};
             \draw[color=black] (2.5,0.7) node {\small $v$};
              \draw[color=black] (0.37,1.2) node {\small $P$};
		\draw[color=black] (0.2,0.6) node {\small $P-d$};
		\draw[color=black] (-0.9,-0.97) node {\small $l$};
			\draw[color=black] (-0.8,-0.5) node {\small $l+\tau$};
			\draw[color=black] (2.8,-2) node {\small $\mathbb{R}^n$};
			
%%%%PROJECTION%%%%

\draw[black, line width = 0.1mm]  (-1,-2.4) -- (2,-2.4);
\draw[black, line width = 0.1mm]  (0,-1.5) -- (3,-1.5);	
\draw[black, line width = 0.1mm]  (-1,-2.4) -- (0,-1.5);		
\draw[black, line width = 0.1mm]  (2,-2.4) -- (3,-1.5);	
\draw[black, line width = 0.1mm, dash pattern=on 0.5pt off 0.5pt]  (1.06,0)  --(1.06,-1.9);	
      \draw [fill=black] (1.06,-1.9) circle (0.4pt); %x_0
                  \draw[color=black] (1.2,-1.8) node {\footnotesize $x_0$};
         \draw [fill=black] (1.015,-0.52-1.6) circle (0.4pt); %x_1
                           \draw[color=black] (1.18,-0.52-1.55) node {\footnotesize $x_1$};
            \draw [fill=black] (0.7,-0.175-1.8) circle (0.4pt); %x_2	
                                       \draw[color=black] (0.7,-0.175-1.7) node {\footnotesize $x_2$};                               
\draw[black, line width = 0.1mm, dash pattern=on 1pt off 1pt, fill=black,fill opacity=0.15]   (1.06,-1.9)  ellipse (0.5cm and 0.3cm);
 \draw[color=black] (1.85,-1.8) node {\footnotesize $B_r(x_0)$};         
\end{tikzpicture}
\caption{Geometry involved in the proof of Proposition~\ref{prop:subsolution}.}
\label{fig:parabolas}
\end{figure}
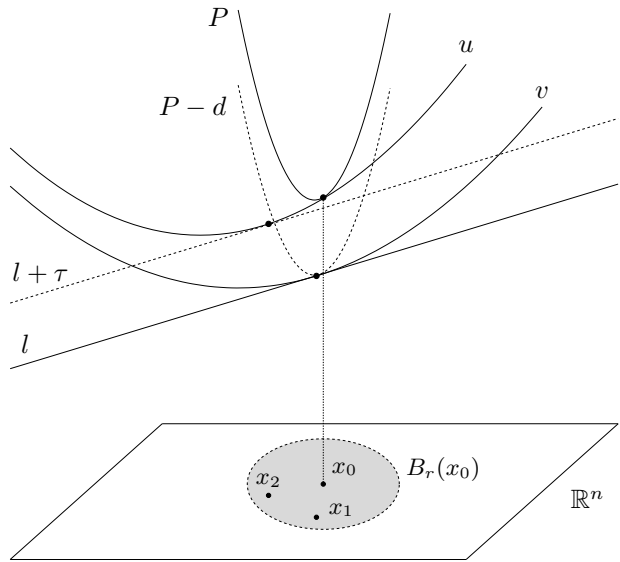

Fix $\vep>0$. By Lemma~\ref{lem:compact}, we have that $D_{x_0} u(\vep)$ is bounded, so $\Lambda =\diam D_{x_0} u(\vep) <\infty$. 
Choose $\delta$ sufficiently small, so that $r<\vep/(4\Lambda)$. Then by Proposition~\ref{prop:lem4.6}, it holds that
\begin{equation}\label{eq:bound1}
\F_k^s u(x_2)  \leq \F_k^s u(x_0)+  C \Lambda^{1-s} |x_2-x_0|^{1-s} + \tfrac{4\Lambda}{\vep}  \F_k^s u(x_0) |x_2-x_0| 
\leq \F_k^s u(x_0) + C(r),
\end{equation}
where $C(r)\to 0$, as $r\to 0$.
Next we will show that
\begin{equation}\label{eq:bound2}
 \F_k^s v(x_1) - C \tau^{1-s} \leq \F_k^s u(x_2)
\end{equation}
for some constant $C>0$ depending only on $n$, $k$, and $s$.
Since $\partial v(x_1) = \{\nabla l\}$, then $v \in C^{1,1}(x_1)$, and using Proposition~\ref{prop:repW}, we get
$$
 \F_k^s v(x_1) = c_{n,s} \int_0^\infty \int_{\R^{n-k}} \frac{1}{|z|^{n-k+2s}} W \Big( \frac{\mu_{x_1} v (t,z)^{1/k}}{|z|}\Big) \, dz dt ,
$$
where $\mu_{x} v (t,z)=\omega_k^{-1} \H^k \big( \{ y \in \R^k :  \tilde{v}_x (y,z) \leq t \} \big)$, and $W$ is the monotone decreasing function given in \eqref{eq:W}.
Observe that since $v\leq u$, $l$ is the supporting plane of $v$ at $x_1$, and $l+\tau$ is the supporting plane of $u$ at $x_2$, then for any $t>0$, it follows that
$$
D_{x_2} u(t) = \{ u-(l+\tau) \leq t\} \subseteq \{ v -l \leq t+\tau\}= D_{x_1} v(t+\tau).
$$
In particular, $\mu_{x_2} u(t,z) \leq \mu_{x_1} v(t+\tau,z)$, for any $z\in \R^{n-k}$. Therefore,
$$
W(\mu_{x_2} u(t,z)) \geq W(\mu_{x_1} v(t+\tau,z)),
$$
which yields
\begin{align*}
 \F_k^s u(x_2) & \geq c_{n,s} \int_\tau^\infty \int_{\R^{n-k}} \frac{1}{|z|^{n-k+2s}} W \Big( \frac{\mu_{x_1} v (t,z)^{1/k}}{|z|}\Big) \, dz dt\\
 &=  \F_k^s v(x_1) - c_{n,s} \int_0^\tau \int_{\R^{n-k}} \frac{1}{|z|^{n-k+2s}} W \Big( \frac{\mu_{x_1} v (t,z)^{1/k}}{|z|}\Big) \, dz dt\\
& \geq  \F_k^s v(x_1) - C \tau^{1-s},
\end{align*}
where the last inequality follows from the fact that $ \mu_{x_1} v(t,z) \geq C(t-|z|^2)_+^{k/2}$ and $W$ is monotone decreasing.

Combining \eqref{eq:bound1} and \eqref{eq:bound2}, using that $v$ is a subsolution, and \eqref{eq:uniformconv}, we get
\begin{align*}
\F_k^s u(x_0) + C(r) & \geq \F_k^s v(x_1)  - C \tau^{1-s} \geq v(x_1) - \varphi(x_1) - C\tau^{1-s} \\
&> u(x_1)-h - \varphi(x_1) - C\tau^{1-s}.
\end{align*}
Letting $\delta\to 0$, it follows that $h\to0$, $C(r)\to 0$, $\tau \to 0$, and $x_1\to x_0$. By continuity of $u$ and $\varphi$, we conclude the result.
\end{proof}

\begin{prop}[$u$ is a supersolution]
The function $u$  given in \eqref{eq:perron}  satisfies 
$$
\F_k^s u (x_0) \leq u(x_0) - \varphi(x_0), \quad \hbox{for all}~x_0 \in \Rn.
$$
\end{prop}

\begin{proof}
Assume the statement is false. Then there exists some $x_0 \in \Rn$ such that
$$
\F_k^s u (x_0) > u(x_0) - \varphi(x_0).
$$
Without loss of generality, we may assume that $u(x_0)=0$ and $\nabla u(x_0)=0$. Otherwise, consider $v(x)= u(x) - u(x_0) -  (x-x_0) \cdot \nabla u(x_0)$. Then there exists some $\delta>0$ such that
\begin{equation} \label{eq:delta}
\F_k^s u (x_0) \geq  - \varphi(x_0) + \delta.
\end{equation}
 
Fix $\vep >0$ and let 
$
u^\vep(x) = \max\{u(x),\vep\}.
$
We will show that for $\vep$ sufficiently small, $u^\vep$ is an admissible subsolution, and thus, reaching a contradiction with $u$ being the largest subsolution. 
Indeed, $u^\vep$ is convex and $u^\vep \in C^{0,1}(\Rn)$ with 
 $[u^\vep]_{C^{0,1}(\Rn)} \leq [\varphi]_{C^{0,1}(\Rn)}.$
Moreover, note that $u^\vep(x) = u(x)$, for $x$ large. Hence, once we show that $u^\vep$ is a subsolution, it will follow from the comparison principle  that $\varphi \leq u^\vep \leq \varphi+w$.

If $x \in \{u_\vep=u\}$, then $u_\vep (x)=u(x)$ and $u_\vep \geq u$ in $\Rn$. By monotonicity (Lemma~\ref{lem:properties}),
\begin{equation*} 
\F_k^s u^\vep (x) \geq\F_k^s u (x) \geq u(x) - \varphi(x) =  u^\vep(x) - \varphi(x),
\end{equation*} 
since $u$ is a subsolution, by Proposition~\ref{prop:subsolution}.

If $x\in \{u^\vep > u\}$, then $u^\vep(x) = \vep$ and $\partial u^\vep(x)=\{0\}$. In particular, 
\begin{equation}\label{eq:constant}
\F_k^s u^\vep (x)= \F_k^s u^\vep(x_0). 
\end{equation}
Moreover, for any $t>0$, we have
$
D_{x_0} u^\vep (t) = \{ u^\vep - \vep \leq t\} = \{u \leq t+\vep \} = D_{x_0} u (t+\vep).
$
Therefore, in view of Proposition~\ref{prop:repW}, we get
\begin{equation} \label{eq:boundep}
 \F_k^s u^\vep(x_0) =  \F_k^s u(x_0) - \int_0^\vep \int_{\R^{n-k}} \frac{1}{|z|^{n-k+2s}} W \Big( \frac{\mu_{x_0} u(t,z)^{1/k}}{|z|}\Big) \, dz dt \geq \F_k^s u(x_0) - C \vep^{1-s}
\end{equation}
since $u\in C^{1,1}(\Rn)$ and $\mu_{x_0} u(t,z) \geq (t-|z|^2)_+^{k/2}$. 

Combining \eqref{eq:delta}, \eqref{eq:constant}, and \eqref{eq:boundep}, we see that
\begin{align*}
\F_k^s u^\vep(x) & = \F_k^s u^\vep(x_0)  \geq  \F_k^s u(x_0) - C \vep^{1-s} \geq - \varphi(x_0) + \delta  - C \vep^{1-s} \\
& =  u^\vep(x)- \varphi(x) + \big ( \varphi(x) - \varphi(x_0) + \delta - C \vep^{1-s} - \vep \big),
\end{align*}
since $u^\vep(x)=\vep$. We need the term inside the parenthesis to be nonnegative. Hence, it remains to control $\varphi(x) - \varphi(x_0)$. Since $\varphi$ is smooth,
$$
|\varphi(x) - \varphi(x_0)|  \leq [\varphi]_{C^{0,1}(\Rn)} |x-x_0|.
$$
We distinguish two cases. If $\{u=0\}=\{x_0\}$, then $|x-x_0| \leq d_\vep\to 0$, as $\vep \to 0$. Hence, choosing $\vep$ sufficiently small, we see that
$$
\varphi(x) - \varphi(x_0) +  \delta - C \vep^{1-s} - \vep \geq \delta - [\varphi]_{C^{0,1}(\Rn)} d_\vep - C \vep^{1-s}-\vep \geq 0.
$$
Therefore,  $u^\vep \in \mathcal{S}$, which contradicts $u^\vep(x_0)> u(x_0) = \sup_{v\in \S} v(x_0) \geq u^\vep(x_0)$.

Suppose now that $\{u=0\}$ contains more than one point. By compactness of $\{u=0\}$ and continuity of $\varphi$, there exists some $x_1 \in \{u=0\}$ where $\varphi$ attains its maximum. 
Then
$$
\F_k^s u(x_1) = \F_k^s u(x_0) \geq u(x_0) - \varphi(x_0) + \delta \geq u(x_1) - \varphi(x_1)+ \delta.
$$
Moreover, by convexity of $\{u=0\}$ (since $u\geq \varphi\geq0$) and $\varphi$, we must have that $x_1\in \partial \{u=0\}$. Hence, there exists $\{x_j\}_{j=2}^\infty \subset \{u>0\}$ such that $x_j\to x_1$ and $u$ is strictly convex at $x_j$. Namely, there is a supporting plane that touches $u$ only at $x_j$. 

By continuity of $u$, there exists some $j_0 \geq 2$ such that
$$
u(x_1) > u(x_j) - \delta/4, \quad \hbox{for all}~j> j_0.
$$
By continuity of $\varphi$, there exists some $j_1\geq 2$ such that
$$
\varphi(x_1) < \varphi(x_j) + \delta/4, \quad \hbox{for all}~j> j_1.
$$
By  lower semicontinuity of $\F_k^s u$, up to a subsequence, there exists some $j_2\geq 2$ such that
$$
\F_k^s u (x_j) > \F_k^s u(x_1) - \delta/4, \quad \hbox{for all}~j> j_2.
$$
Let $J> \max\{j_0,j_1,j_2\}$. Then
\begin{align*}
\F_k^s u(x_J) & > \F_k^s u( x_1) - \delta/4 \geq u(x_1) - \varphi(x_1) + 3\delta/4> u(x_J) - \varphi(x_J) + \delta/4,
\end{align*}
and we can repeat the previous argument, replacing $x_0$ by $x_J$.
We conclude that
\begin{equation*}
\F_k^s u (x_0) \leq u(x_0) - \varphi(x_0), \quad\hbox{for all}~x_0 \in \Rn.
\end{equation*}
\end{proof}

%%%%%%%%%%%%%%%%%%%%%%%%%%%%%%%%%%%%%
\section{Future directions} \label{sec:continuous}
%%%%%%%%%%%%%%%%%%%%%%%%%%%%%%%%%%%%%

 As mentioned in the introduction, the main idea to define a nonlocal analog to the Monge-Amp\`{e}re operator is to write it as a concave envelope of linear operators. More precisely,
\begin{equation*}
n \det(D^2 u(x))^{1/n} = \inf_{M \in \M} \tr(M D^2 u(x) ),
\end{equation*}
where $\M=\{ M \in \S^n: M>0,  \ \det(M)=1\}$ and $\S^n$ is the set of $n\times n$ symmetric matrices.
Note that this identity is equivalent to the one given in \eqref{eq:MAidentity} taking $M=AA^T$ and $B=D^2 u(x)$, since $\tr(A^TBA)=\tr(AA^T B)$.
In fact, this extremal property does not only hold for $n\det(B)^{1/n}$ with $B\in\S^n$ and $B>0$. 
If $\lambda=(\lambda_1,\dots,\lambda_n)$, where $\lambda_i$ are the eigenvalues of $B$, then the function $f$ defined in $\Gamma= \{\lambda \in \Rn : \lambda_i >0, \ \text{for all}~i=1,\hdots, n\}$, given by
$$
f(\lambda)= n \Big( \prod_{i=1}^n \lambda_i \Big)^{1/n}=n\det(B)^{1/n}
$$
 is differentiable, concave, and homogeneous of degree 1. In general, if $f$ satisfies these conditions in an open convex set $\Gamma$ in $\Rn$, then 
\begin{equation*}
f(\lambda) = \inf_{\mu \in \Gamma}\big \{ f(\mu) + \nabla f(\mu) \cdot (\lambda-\mu) \big\} 
= \inf_{\mu \in \Gamma} \nabla f(\mu) \cdot \lambda,
\end{equation*}
where the second identity follows by Euler's theorem. Therefore,
$$
f(\lambda) = \inf_{M\in \M_f} \tr(M B),
$$
where $\M_f = \{ M \in \S^n :  \lambda(M) \in \nabla f(\Gamma)\}$, $\nabla f (\Gamma) = \{ \nabla f (\mu) : \mu \in \Gamma\}$, and $\lambda(M)$ are the eigenvalues of $M$.

For instance, the $k$-Hessian functions introduced by Caffarelli--Nirenberg--Spruck in \cite{CNS} satisfy these conditions and, in fact, fractional analogs have been recently studied by Wu \cite{Wu}. It would be interesting to explore fractional analogs to a wider class of fully nonlinear concave operators, as the ones mentioned above.

We remark that the $1$-Hessian is equal to the Laplacian, and the $n$-Hessian is equal to the Monge-Amp\`{e}re operator. Moreover,  for $1<k<n$, we obtain an intermediate \textit{discrete} family between these operators.
In view of this observation, a natural question of finding a \textit{continuous} family connecting the Laplacian with the Monge-Amp\`{e}re operator arises. 
Here we suggest possible families that connect smoothly these two operators, passing through the $k$-Hessians, in some sense.
Indeed, let $\alpha \in (0,1]^n$ and denote $|\alpha|=\alpha_1 + \cdots +\alpha_n$. For $\lambda\in \Rn_+$, we consider the functions,
$$
f_\alpha (\lambda) =\Big( \sum_{\sigma \in S} \lambda_{\sigma(1)}^{\alpha_1} \cdots \lambda_{\sigma(n)}^{\alpha_n} \Big)^\frac{1}{|\alpha|},
$$
where $S$ is the set of all cyclic permutations of $\{1, \hdots, n\}$.
Observe that for any $1\leq k \leq n$, if $\alpha = \sum_{i \in \mathcal{I}} e_i$, with $|\mathcal{I}|=k$, then $f_\alpha$ is precisely the $k$-Hessian function. Consider any smooth simple curve $\gamma : [0,1] \to (0,1]^n$ such that 
\begin{enumerate}
\item $\gamma(0)=e_i$, for some $1\leq i\leq n$,
\item $\gamma(t_k) = \sum_{i \in \mathcal{I}_k} e_i$, with $|\mathcal{I}_k|=k$, and $0<t_k< t_{k+1} <1$, for all $1 < k <n$, and
\item $\gamma(1) = (1,\hdots,1)$. 
\end{enumerate}
Then the family $\{f_\alpha\}_{\alpha\in \operatorname{Im}(\gamma)}$ is as we described.
In particular, fractional analogs of these functions would give a continuous family from the fractional Laplacian to the nonlocal Monge-Amp\`{e}re. We will study this problem in a forthcoming paper.

\medskip

\noindent\textbf{Acknowledgments.} We would like to thank the referees for some important comments and clarifications, which have greatly improved the final exposition of this work.

%%%%%%%%%%%%%%%%%%%%%%%%%%%%%%%%%%%%%%%%%%%%%

\end{document}